\documentclass{amsart}
\usepackage{amssymb}

\newcommand{\Q}{\mathbb{Q}}
\newcommand{\C}{\mathbb{C}}

\newcommand{\N}{\mathbb{N}}

\newcommand{\cat}{^\frown}
\newcommand{\dom}{\operatorname{dom}}
\newcommand{\ran}{\operatorname{ran}}

\newcommand{\norm}[1]{\left\| #1 \right\|}
\newcommand{\supp}{\operatorname{supp}}

\theoremstyle{theorem}
\newtheorem{theorem}{Theorem}[section]
\newtheorem{lemma}[theorem]{Lemma}

\theoremstyle{definition}
\newtheorem{definition}[theorem]{Definition}

\theoremstyle{theorem}
\newtheorem{corollary}[theorem]{Corollary}

\theoremstyle{theorem}
\newtheorem{proposition}[theorem]{Proposition}

\theoremstyle{theorem}

\theoremstyle{theorem}

\theoremstyle{definition}

\theoremstyle{theorem}

\numberwithin{equation}{section}

\begin{document}
\title{Analytic computable structure theory and $L^p$ spaces.}

\author{Joe Clanin}
\address{Department of Computer Science\\
Iowa State University\\
Ames, Iowa 50011}
\email{jsc@iastate.edu}

\author{Timothy H. McNicholl}
\address{Department of Mathematics\\
Iowa State University\\
Ames, Iowa 50011 USA}
\email{mcnichol@iastate.edu}

\author{Don M. Stull}
\address{Department of Computer Science\\
Iowa State University\\
Ames, Iowa 50011 USA}
\email{dstull@iastate.edu}
\address{Laboratoire lorrain de recherche en informatique et ses applications\\
Campus scientifique\\
BP 239\\
54506 Vandoeuvre-l\'es Nancy Cedex\\
FRANCE
}
\email{donald.stull@inria.fr}

\begin{abstract}
We continue the investigation of analytic spaces from the perspective of computable structure theory.  
We show that if $p \geq 1$ is a computable real, and if $\Omega$ is a nonzero, non-atomic, and separable measure space, then every computable presentation of $L^p(\Omega)$ is computably linearly isometric to the standard computable presentation of $L^p[0,1]$; in particular, $L^p[0,1]$ is computably categorical.  
We also show that there is a measure space $\Omega$ that does not have a computable presentation even though 
$L^p(\Omega)$ does for every computable real $p \geq 1$.  
\end{abstract}
\maketitle

\section{Introduction}\label{sec:intro}

In 1961, A.N. Mal'cev, motivated by the work of Fr\"olich, Shepherdson and others on effective field theory, set forth the idea of an \emph{effective numbering} of an algebra \cite{Malcev.1961}; these are now called \emph{computable presentations}.    Specifically, a computable presentation of a structure (such as a group, ring, graph, etc.) is an assignment of nonnegative integers to the elements of the domain so that the induced relations (including the relation induced by equality) and functions on the nonnegative integers are computable.  
A computable presentation of a structure can be thought of as a way of imposing a notion of computability on the structure in that it induces a set of computable functions and relations on the structure.  
A structure is \emph{computably presentable} if it has a computable presentation.  So, we can think of the computably presentable structures as those upon which we can compute.  Computable structure theory is the study of computable presentations of mathematical structures.

Fr\"olich and Shepherdson were the first to notice that different computable presentations of a structure may yield different classes of computable sets and operations.  Specifically, they demonstrated that there is a field for which there exist two computable presentations so that a splitting algorithm exists with respect to the first computable presentation but not with respect to the second  \cite{Froehlich.Shepherdson.1956}.   
Accordingly, Mal'cev defined a computably presentable structure to be \emph{autostable} if any two of its computable presentations are computably isomorphic \cite{Malcev.1962}.  Autostable structures are more commonly referred to as \emph{computably categorical}.  Thus, a computably categorical structure can be thought of as one for which there is an absolute notion of computability; for all other structures, computability is referent to a computable presentation. 

The computable categoricity of structures in various algebraic and combinatorial classes (e.g. countable linear orders, groups, graphs, etc.) has been intensively studied.  
However, until recently, analytic structures such as metric spaces and Banach spaces have been ignored in this context. Thus, a research program has lately emerged to apply computable structure theory to analytic spaces.  
One obvious obstacle is that these spaces are generally uncountable.  However, our understanding of computability on analytic spaces has advanced considerably in the last few decades and should no longer be seen 
as an impediment.  

Here we focus on the computable categoricity of $L^p$ spaces due to their centrality in many branches of analysis and computational mathematics.  Indeed, it could be argued that these spaces are more relevant for most mathematicians than the typical countable structures commonly studied in computable structure theory.  It is generally agreed that computability can only be studied on separable spaces (at least with our current understanding of computation).  If an $L^p$ space is separable, then its underlying measure space is separable.  
Thus, we restrict our attention to $L^p$ spaces of separable measure spaces.  The computably categorical $\ell^p$ spaces have been classified \cite{McNicholl.2015}, \cite{McNicholl.2016}.  So, here we will focus on non-atomic spaces.  Our main theorem is the following.

\begin{theorem}\label{thm:main}
If $\Omega$ is a nonzero, non-atomic, and separable measure space, and if $p \geq 1$ is a computable real, then every computable presentation of $L^p(\Omega)$ is computably isometrically isomorphic to the standard computable presentation of $L^p[0,1]$.
\end{theorem}

Note that when we say that a measure space $\Omega = (X, \mathcal{M}, \mu)$ is nonzero, we mean that there is a set $A \in \mathcal{M}$ so that $0 < \mu(A) < \infty$ (so that $L^p(\Omega)$ is nonzero).  

There are several corollaries. 

\begin{corollary}\label{cor:Lp[0,1]}
If $\Omega$ is a non-atomic and separable measure space, and if $p \geq 1$ is a computable real, then $L^p(\Omega)$ is computably categorical.  In particular, for every computable real $p \geq 1$, $L^p[0,1]$ is computably categorical.
\end{corollary}

Thus, when $p \geq 1$ is computable, $L^p[0,1]$ possesses an absolute notion of computability and we need not concern ourselves about which computable presentation we choose when studying its computability theory.

\begin{corollary}\label{cor:Lp.iso}
Let $p$ be a computable real so that $p \geq 1$, and suppose $\Omega_1$, $\Omega_2$ are measure spaces that are nonzero, non-atomic, and separable.  Then, each computable presentation of 
$L^p(\Omega_1)$ is computably isometrically isomorphic to each computable presentation of $L^p(\Omega_2)$.
\end{corollary}

\begin{corollary}\label{cor:msr}
If $\Omega^\#$ is a computable presentation of a nonzero and non-atomic measure space $\Omega$, 
and if $p \geq 1$ is a computable real, then the induced computable presentation of $L^p(\Omega)$ is computably isometrically isomorphic to $L^p[0,1]$.
\end{corollary}

Previously, the second author showed that $\ell^p_n$ is computably categorical when $p \geq 1$ is a computable real.  This provided the first non-trivial example of a computably categorical Banach space that is not a Hilbert space.  Corollary \ref{cor:Lp[0,1]} provides the first example of a computably presentable and infinite-dimensional Banach space that is computably categorical but not a Hilbert space.

Our main theorem can be seen as an effective version of a result of Carath\'eodory: if $\Omega$ is a measure space that is nonzero, non-atomic, and separable, then $L^p(\Omega)$ is isometrically isomorphic to $L^p[0,1]$ \cite{Cembranos.Mendoza.1997}.  
However, our proof is not a mere effectivization of a classical proof.  For, the classical proofs of Carath\'eodory's result all begin with a sequence of transformations on the underlying measure space.  Specifically, it is first shown that there is a $\sigma$-finite measure space $\Omega_1$ so that $L^p(\Omega)$ is isometrically isomorphic to 
$L^p(\Omega_1)$.  It is then shown that there is a probability space $\Omega_2$ so that $L^p(\Omega_2)$ is isometrically isomorphic to $L^p(\Omega_1)$ is then shown to be isometrically isomorphic to $L^p[0,1]$.  This approach is the natural course to take in the classical setting wherein one has full access to the $L^p$ space and to the underlying measure space.  But, in the world of effective mathematics, a computable presentation of $L^p(\Omega)$ does not yield a computable presentation of the underlying measure space; i.e. it allows us to `see' the vectors but not necessarily the measurable sets.  This point will be made precise by way of an example in Section \ref{sec:comp.msr.spaces}.  In particular, Theorem \ref{thm:main} is a stronger result than Corollary \ref{cor:msr}.  Thus our proof of Theorem \ref{thm:main} yields a new proof of Carath\'eodory's result that does not make any transformations on the underlying measure space.  Our main tool for doing this is the concept of a disintegration of an $L^p(\Omega)$ space which was previously used for $\ell^p$ spaces by the second author but which we introduce here for arbitrary $L^p$ spaces.  

The paper is organized as follows.  Section \ref{sec:background} presents background and preliminaries from 
analysis and computability theory; in particular it gives a very brief survey of computable structure theory in the countable setting and a summary of prior results on analytic computable structure theory.  More expansive surveys of classical computable structure theory can be found in \cite{Fokina.Harizanov.Melnikov.2014} and \cite{Harizanov.1998}.  Section \ref{sec:overview} gives an overview of the proof of Theorem \ref{thm:main}.  
In Section \ref{sec:classical}, we develop precursory new material from classical analysis, in particular on disintegrations.  Section \ref{sec:computable} contains our new results on computable analysis and forms the bridge from the classical material in Section \ref{sec:classical} to Theorem \ref{thm:main}.  Section \ref{sec:comparison} contrasts our methods with those used for $\ell^p$ spaces.  Section \ref{sec:rcc} explores relative computable categoricity of $L^p$ spaces.  Results on computable measure spaces and related $L^p$ spaces are expounded in Section \ref{sec:comp.msr.spaces}.  
Section \ref{sec:conclusion} gives concluding remarks.\footnote{After submission, the authors became aware of the work of F. Steinberg on representations of $L^p$ spaces in \cite{Steinberg.2017}.  Our main theorem can be seen as an extension of Theorem 3.8 of his paper.}

\section{Background and preliminaries}\label{sec:background}

We first cover preparatory material from classical mathematics after which we summarize preliminaries 
from computable (effective) mathematics.  In each case we summarize relevant standard information and content specific to this paper.  We then briefly survey the background of classical (i.e. countable) computable structure theory and prior results in analytic computable structure theory.  

\subsection{Classical world}\label{subsec:back.classical}

We begin with a few preliminaries from discrete mathematics.  We then cover preliminaries from 
measure theory and Banach spaces (in particular, $L^p$ spaces).

\subsubsection{Discrete preliminaries}

When $A$ is a finite set, we denote its cardinality by $\# A$.

When $\mathbb{P} = (P, \leq)$ is a partial order and $a,b \in P$, we write $a | b$ if $a,b$ are incomparable; i.e. if $a \not \leq b$ and $b \not \leq a$.  A lower semilattice $(\Lambda, \leq)$ is \emph{simple} if $\mathbf{0}$ is the meet of any two incomparable elements of $\Lambda$.   A lower semilattice $\Lambda'$ is a \emph{proper extension} of a lower semilattice $\Lambda$ if 
$\Lambda \subset \Lambda'$ and for every $u \in \Lambda' - \Lambda$ there is no nonzero $v \in \Lambda$ so that $u > v$.

Suppose $\mathbb{P}_0 = (P_0, \leq_0)$ and $\mathbb{P}_1 = (P_1, \leq_1)$ are partial orders.  
A map $f : P_0 \rightarrow P_1$ is \emph{monotone} if $f(a) \leq_1 f(b)$ whenever $a \leq_0 b$ and is \emph{antitone} if $f(b) \leq_1 f(a)$ whenever $a \leq_0 b$ \; \cite{Chajda.Halavs.Radomir.Kuhr.2007}.

$\N$ denotes the set of all nonnegative integers.  $\N^*$ denotes the set of all finite sequences
of nonnegative integers.  (We regard a sequence as a map whose domain is an initial segment of $\N$.)  These sequences are referred to as \emph{nodes} and the empty sequence $\lambda$ is referred to as the \emph{root node}.  When $\nu \in \N^*$, $|\nu|$ denotes the length of $\nu$ (i.e. the cardinality of the domain of $\sigma$).  When $\nu, \nu' \in \N^*$, we write 
$\nu \subset \nu'$ if $\nu$ prefixes $\nu'$; in this case we also say that $\nu$ is an \emph{ancestor} of $\nu'$ and that $\nu'$ is a \emph{descendant} of $\nu$.  Thus, $(\N^*, \subseteq)$ is a partial order.  When $\nu, \nu' \in \N^*$, write 
$\nu\cat\nu'$ for the concatenation of $\nu$ with $\nu'$.  We say that $\nu'$ is a \emph{child} of $\nu$ if 
$\nu' = \nu\cat(n)$ for some $n \in \N$ in which case we also say that $\nu$ is the \emph{parent} of $\nu'$.  
If $\nu$ is a node, then $\nu^+$ denotes the set of all children of $\nu$ and if $\nu$ is a non-root node then 
$\nu^-$ denotes the parent of $\nu$.  We denote the lexicographic order of $\N^*$ by $<_{\rm lex}$.

If $S$ is a set of nodes, then $\nu \in S$ is \emph{terminal} if $\nu^+ \cap S = \emptyset$.

By a \emph{tree} we mean a set $S$ of nodes so that each ancestor of a node of $S$ also belongs to $S$; i.e. $S$ is closed under prefixes.   
A set $S$ of nodes is an \emph{orchard} if it contains all of the non-root ancestors of each of 
its nodes and does not contain the root node; equivalently, if $\emptyset \not \in S$ and $S \cup \{\emptyset\}$ is a tree.  Note that if $(\Lambda, \leq)$ is a finite simple lower semilattice, then $(\Lambda - \{\mathbf{0}\}, \geq)$ is isomorphic to an orchard.

\subsubsection{Measure-theoretic preliminaries}

We begin by summarizing relevant facts about separable measure spaces and atoms.

Suppose $\Omega = (X, \mathcal{M}, \mu)$ is a measure space.  A collection $\mathcal{D} \subseteq \mathcal{M}$ of sets whose measures are all finite is \emph{dense in $\Omega$} if for every $A \in \mathcal{M}$ with finite measure and every $\epsilon > 0$ there exists $D \in \mathcal{D}$ so that $\mu(D \triangle A) < \epsilon$.  A measure space is \emph{separable} if it has a countable dense set of measurable sets.  

A measurable set $A$ of a measure space $\Omega$ is an \emph{atom} of $\Omega$ if $\mu(A) >0$ and if there is no measurable subset $B$ of $A$ so that $0 < \mu(B) < \mu(A)$.  If $\Omega$ has no atoms, it is said to be \emph{non-atomic}.   The following is due to Sierpinski \cite{Sierpinski.1922}.

\begin{theorem}\label{thm:atomic}
Suppose $\Omega$ is a non-atomic measure space.  
Then, whenever $A$ is a measurable set and $0 < r < \mu(A) < \infty$, there is a 
measurable subset $B$ of $A$ so that $\mu(B) = r$.
\end{theorem}

We will also use the following observation.

\begin{proposition}\label{prop:abs.cont}
Every finite measure that is absolutely continuous with respect to a non-atomic measure is itself non-atomic.
\end{proposition}

\begin{proof}
Suppose $\Omega = (X, \mathcal{M}, \mu)$ is a non-atomic measure space, and let $\nu$ be a finite measure that is absolutely continuous with respect to $\mu$.  

We first claim that whenever $A$ is a measurable set so that $\nu(A) > 0$, there is a measurable subset $B$ of $A$ so that $\mu(B) < \infty$ and $\nu(B) > 0$.  For, let $f = d\nu/d\mu$.  Since $\nu$ is finite, $f$ is integrable.  Since $\nu(A) > 0$, there is a simple function $s$ so that $0 \leq s \leq f$ and $\int_A s\ d\mu > 0$.  Let 
$B_a = s^{-1}[\{a\}] \cap A$ for each real number $a$.  
Since $0 < \int_A s\ d\mu < \infty$, it follows that $0 < \mu(B_a) < \infty$ for some positive real $a$.  
Then,
\begin{eqnarray*}
\nu(B_a) & = & \int_{B_a} f\ d\mu\\
& \geq & \int_{B_a} s\ d\mu\\
& = & a \mu(B_a) > 0.
\end{eqnarray*}

Now, let $A$ be a measurable set so that $\nu(A) > 0$.   We show that $A$ is not an atom of $\nu$.  
Choose a measurable subset $B$ of $A$ so that $\nu(B) > 0$ and $\mu(B) < \infty$.  It suffices to show that 
$B$ is not an atom.  By way of contradiction, suppose it is.  We define a descending sequence of measurable subsets of $B$ as follows.  Set $B_0 = B$.  
Suppose $B_n$ has been defined, $\mu(B_n) > 0$, $\nu(B_n) = \nu(B)$, and $\mu(B_n) = 2^{-n} \mu(B)$.  Since $\Omega$ is non-atomic, 
by Theorem \ref{thm:atomic}, there is a measurable subset $C$ of $B_n$ so that $\mu(C) = \frac{1}{2}\mu(B_n)$.  Let $D = B_n - C$.  
Since $B_n \subseteq B$, $B_n$ is an atom of $\nu$.  
Thus, either $\nu(C)$ or $\nu(D)$ is equal to $\nu(B_n)$; without loss of generality, assume $\nu(C) = \nu(B_n)$.  Set $B_{n+1} = C$.  Let $B' = \bigcap_n B_n$.  Thus, $\mu(B') = 0$.  On the other hand, $\nu(B') = \lim_n \nu(B_n) = \nu(B) \neq 0$, and so we have a contradiction since $\nu$ is absolutely continuous with respect to $\mu$.  Thus, $B$ is not an atom of $\nu$.
\end{proof}

We identify measurable sets whose symmetric difference is null.  When we refer to a collection of measurable sets as a lower semilattice, we mean it is a lower semilattice under the partial ordering of inclusion modulo sets of measure $0$.

\subsubsection{Banach space preliminaries}

We first cover material relevant to Banach spaces in general and then that which is specific to $L^p$ spaces.  We take the complex numbers to be the field of scalars however all 
of our results hold in the real case as well.

Suppose $\mathcal{B}$ is a Banach space.  When $X \subseteq \mathcal{B}$, we write 
$\mathcal{L}(X)$ for the linear span of $X$ and 
$\langle X \rangle$ for the closed linear span of $X$; i.e. $\langle X \rangle = \overline{\mathcal{L}(X)}$.
When $K$ is a subfield of $\C$, write $\mathcal{L}_K(X)$ for the linear span of $X$ over $K$; i.e. 
\[
\mathcal{L}_K(X) = \{ \sum_{j = 0}^M \alpha_j v_j\ :\ M \in \N\ \wedge\ \alpha_0, \ldots, \alpha_M \in K\ \wedge\ v_0, \ldots, v_M \in X\}.
\]
Note that the linear span of $X$ is dense in $\mathcal{B}$ if and only if the linear span of $X$ over $\Q(i)$ is dense in $\mathcal{B}$.

When $S$ is a finite set, we let $\mathcal{B}^S$ denote the 
set of all maps from $S$ into $\mathcal{B}$.  When $f \in \mathcal{B}^S$, we write 
$\norm{f}_S$ for $\max\{\norm{f(t)}\ :\ t \in S\}$.  It follows that $\norm{\ }_S$ is a norm on $\mathcal{B}^S$ 
under which $\mathcal{B}^S$ is a Banach space.  

Computability on Banach spaces will be defined in terms of structures and presentations.  Although these
notions may be germane only to computability theory, they are nevertheless purely classical objects so we cover them and related concepts here.
A \emph{structure} on $\mathcal{B}$ is a map $D : \N \rightarrow \mathcal{B}$ so that $\mathcal{B} = \langle \ran(D) \rangle$.  If $D$ is a structure on $\mathcal{B}$, then we call the pair $(\mathcal{B}, D)$ a \emph{presentation} of $\mathcal{B}$.  Clearly, a Banach space has a presentation if and only if it is separable.  
Among all presentations of a Banach space $\mathcal{B}$, one may 
be designated as \emph{standard}; in this case, we will identify $\mathcal{B}$ with its standard presentation.
In particular, if $p \geq 1$ is a computable real, and if $D$ is a standard map of $\N$ onto the set of 
characteristic functions of dyadic subintervals of $[0,1]$, then $(L^p[0,1], D)$ is the standard presentation of $L^p[0,1]$.  If $R(n) = 1$ for all $n \in \N$, then $(\C, R)$ is the standard presentation of $\C$ as a Banach space over itself.

Each presentation of a Banach space induces corresponding classes of rational vectors and rational open balls as follows. Suppose $\mathcal{B}^\# = (\mathcal{B}, D)$ is 
a presentation of $\mathcal{B}$.  Each vector in the linear span of $\ran(D)$ over $\Q(i)$ will be called a 
\emph{rational vector of $\mathcal{B}^\#$}.  An \emph{open rational ball of $\mathcal{B}^\#$} is an open ball whose center is a rational vector of $\mathcal{B}^\#$ and whose radius is a positive rational number. 

A presentation of $\mathcal{B}$ induces a corresponding presentation of $\mathcal{B}^S$ as follows.  Suppose $\mathcal{B}^\# = (\mathcal{B}, D)$ is a presentation of $\mathcal{B}$.  Let $S$ be a finite set, and let $D^S$ denote a standard map of $\N$ onto 
the set of all maps from $S$ into $\ran(D)$.  It follows that $(\mathcal{B}^S)^\# := (\mathcal{B}^S, D^S)$ 
is a presentation of $\mathcal{B}^S$.   

We now cover preliminaries on $L^p$ spaces.  Fix a measure space $\Omega = (X, \mathcal{M}, \mu)$ and a real $p \geq 1$.  When $f \in L^p(\Omega)$, 
we write $\supp(f)$ for the support of $f$; i.e. the set of all $t \in X$ so that $f(t) \neq 0$.  Note that since we identify measurable sets whose symmetric difference is null, $\supp(f)$ is well-defined.  
We say that vectors $f,g \in L^p(\Omega)$ are \emph{disjointly supported} if the intersection of their supports is null; i.e. if $f(t)g(t) = 0$ for almost every $t \in X$.  
If $f,g \in L^p(\Omega)$, then we write $f \preceq g$ if $f(t) = g(t)$ for almost all $t \in X$ for which $f(t) \neq 0$; in this case we say that $f$ is a \emph{subvector} of $g$.  Note that $f$ is a subvector of $g$ if and only if $g - f$ and $f$ are disjointly supported.  Note also that $f \preceq g$ if and only if 
$f = g \cdot \chi_A$ for some measurable set $A$.  

When we refer to a collection $\mathcal{D} \subseteq L^p(\Omega)$ as a 
lower semilattice, we mean it is a lower semilattice with respect to the subvector ordering.

Suppose $S$ is a set of nodes and $\phi : S \rightarrow L^p(\Omega)$.  We say that $\phi$ is \emph{separating} if it maps incomparable nodes to disjointly supported vectors.   

We now formulate a numerical test for disjointness of support.  Suppose $p \geq 1$ and $p \neq 2$.  When $z, w \in \C$ let:
\[
\sigma(z, w) = |4 - 2\sqrt{2}^p |^{-1} |2 (|z|^p + |w|^p) - (|z - w|^p + |z + w|^p)|
\]
We will use the following result from \cite{McNicholl.2016} which extends a theorem of J. Lamperti \cite{Lamperti.1958}.

\begin{theorem}\label{thm:lamperti}
Suppose $p \geq 1$ and $p \neq 2$.  
\begin{enumerate}
	\item For all $z,w \in \C$, \label{thm:lamperti::itm:ineq}
\[
\min\{|z|^p, |w|^p\} \leq \sigma(z,w). 
\]

	\item Furthermore, if $1 \leq p < 2$, then \label{thm:lamperti::itm:sign}
\[
2|z|^p + 2|w|^p - |z+w|^p - |z-w|^p \geq 0
\]
and if $2 < p$ then 
\[
2|z|^p + 2|w|^p - |z+w|^p - |z-w|^p \leq 0.
\]
\end{enumerate}
\end{theorem}

Again, suppose $p \geq 1$ and $p \neq 2$.  Let $\Omega = (X, \mathcal{M}, \mu)$ be a measure space.  
When $f,g \in L^p(\Omega)$, let 
\[
\sigma(f,g) = |4 - 2\sqrt{2}^p |^{-1} |2 (\norm{f}^p_p + \norm{g}_p^p) - (\norm{f - g}_p^p + \norm{f + g}_p^p)|
\]
It follows from Theorem \ref{thm:lamperti}.\ref{thm:lamperti::itm:sign} that 
\[
\sigma(f,g) = \int_X \sigma(f(t), g(t))\ d\mu(t).
\]
It then follows that $f,g$ are disjointly supported if and only if $\sigma(f,g) = 0$.

When $S$ is a finite set and $\psi : S \rightarrow L^p(\Omega)$, set 
\[
\sigma(\psi)  =  \sum_{\nu | \nu'} \sigma(\psi(\nu), \psi(\nu')) + \sum_{\nu' \supset \nu} \sigma(\psi(\nu') - \psi(\nu), \psi(\nu'))
\]
where $\nu$, $\nu'$ range over $S$.  Theorem \ref{thm:lamperti} now yields the following numerical test to see if a map is separating and antitone.

\begin{corollary}\label{cor:sigma}
Suppose $1 \leq p < \infty$ and $p \neq 2$.  Suppose $S$ is a finite set of nodes and $\phi : S \rightarrow L^p(\Omega)$.  Then, $\phi$ is a separating antitone map if and only if $\sigma(\phi) = 0$.
\end{corollary}

Now, suppose $S$ is a tree.  Call a map $\phi : S \rightarrow L^p(\Omega)$ \emph{summative} if 
\[
\phi(\nu) = \sum_{\nu' \in \nu^+ \cap S} \phi(\nu')
\]
 whenever $\nu$ is a nonterminal node of $S$.    A \emph{disintegration} is a summative, separating, and antitone map $\phi : S \rightarrow L^p(\Omega) - \{\mathbf{0}\}$ with the additional property that the linear span of its range is dense in $L^p(\Omega)$.
We define a  \emph{partial disintegration of $L^p(\Omega)$} to be a separating and injective antitone map 
of a finite orchard into $L^p(\Omega) - \{\mathbf{0}\}$.

Now, suppose $\Omega_1$ and $\Omega_2$ are measure spaces.  
Suppose $\phi_1$, $\phi_2$ are antitone maps of $L^p(\Omega_1)$ and $L^p(\Omega_2)$ respectively.  An \emph{isomorphism} of $\phi_1$ with $\phi_2$ is 
an injective monotone map $f$ of $\dom(\phi_1)$ onto $\dom(\phi_2)$ so that 
$\norm{\phi_2 (f(\nu))}_p = \norm{\phi_1(\nu)}_p$ for all $\nu \in \dom(\phi_1)$.

A map $\phi : S \rightarrow L^p[0,1]$ is \emph{interval-valued} if $\phi(\nu)$ is the characteristic function of an interval for each $\nu \in \dom(\phi)$.

\subsection{Computable world}\label{subsec:computable}

We assume the reader is familiar with the rudiments of computability theory such as computable functions, sets, c.e. sets, and oracle computation.  An excellent reference is \cite{Cooper.2004}.

\subsubsection{Computable categoricity in the countable realm}

To give our work some context we synopsize some background material on computable structure theory
in the countable realm; this will motivate our definitions for Banach spaces below as well as some already given.  In particular we give precise definitions of \emph{computable categoricity} and \emph{relative computable categoricity} and survey related results. More expansive expositions can be found in \cite{Fokina.Harizanov.Melnikov.2014} and \cite{Ash.Knight.2000}.  

To begin, suppose $\mathcal{A}$ is a structure with domain $A$.  A \emph{numbering} of $\mathcal{A}$ is a surjection of $\N$ onto $A$.  If $\nu$ is a numbering of $\mathcal{A}$, then the pair $(\mathcal{A}, \nu)$ is called a 
\emph{presentation} of $\mathcal{A}$.  Suppose $\mathcal{A}^\# = (\mathcal{A}, \nu)$ is a presentation of $\mathcal{A}$.  
We say that $\mathcal{A}^\#$ is a \emph{computable presentation} of $\mathcal{A}$ if:
\begin{itemize}
	\item $\{(m,n)\ :\ \nu(m) = \nu(n)\}$ is computable, 

	\item for each $n$-ary relation $R$ of $\mathcal{A}$, $\{(x_1, \ldots, x_n)\ :\ R(\nu(x_1), \ldots, \nu(x_n)\}$ is computable, and 
	
	\item for each $n$-ary function $f : A^n \rightarrow A$ of $\mathcal{A}$, 
	$\{(x_1, \ldots, x_n, y)\ :\ f(\nu(x_1), \ldots, \nu(x_n)) = \nu(y)\}$ is computable.
\end{itemize}
Note we regard constants as $0$-ary functions.  
We say that a countable structure $\mathcal{A}$ is \emph{computably presentable} if it has a computable presentation.  It is well-known that there are countable structures without computable presentations; see \cite{Fokina.Harizanov.Melnikov.2014} for a survey of such results.
  
Suppose $\mathcal{A}_1$ and $\mathcal{A}_2$ are structures, and suppose $\mathcal{A}_j^\# = (\mathcal{A}_j, \nu_j)$ is a presentation of $\mathcal{A}_j$ for each $j$.  We say that a map 
$f : \mathcal{A}_1 \rightarrow \mathcal{A}_2$ is a \emph{computable map of $\mathcal{A}_1^\#$ into $\mathcal{A}_2^\#$} if there is a computable map $F : \N \rightarrow \N$ so that $f(\nu_1(n)) = \nu_2(F(n))$ for all $n \in \N$.  We similarly define what it means for an oracle to compute a map of $\mathcal{A}_1^\#$ into $\mathcal{A}_2^\#$.

We say that a computably presentable countable structure $\mathcal{A}$ is \emph{computably categorial} if any two computable presentations of $\mathcal{A}$ are computably isomorphic.  This is equivalent to saying that $\mathcal{A}_1^\#$ is computably isomorphic to $\mathcal{A}_2^\#$ whenever $\mathcal{A}_1^\#$ and $\mathcal{A}_2^\#$ are computable presentations of structures that are isomorphic to $\mathcal{A}$.  

It is easy to see that $(\Q, <)$ is computably categorical (use Cantor's back-and-forth construction).  On the other hand, a fairly straightforward diagonalization shows that $(\N, <)$ is not computably categorical.

As mentioned in the introduction, the interaction of computable categoricity and structure has been 
studied extensively.  
For example, J. Remmel showed that a computably presentable Boolean algebra is computably categorical if and only it is has finitely many atoms \cite{Remmel.1981.2}.  Goncharov, Lempp, and Solomon proved that a computably presentable ordered Abelian group is computably categorical if and only if it has finite rank \cite{Goncharov.Lempp.Solomon.2003}.  Recently, O. Levin proved that every computably presentable ordered field with finite transcendence degree is computably categorical \cite{Levin.2016}.
The effect of structure on other computability notions has been studied intensively; see e.g. \cite{Hirschfeldt.Khoussainov.Shore.Slinko.2002} for a very good overview.

We now define relative computable categoricity.  We first define the diagram of a presentation.  Suppose $\mathcal{A}$ is a structure and $\mathcal{A}^\# = (\mathcal{A}, \nu)$ is a presentation of $\mathcal{A}$.  
The \emph{diagram} of $\mathcal{A}^\#$ is the join of the following sets. 
\begin{itemize}
	\item $\{(m,n)\ :\ \nu(m) = \nu(n)\}$.

	\item $\{(x_1, \ldots, x_n)\ :\ R(\nu(x_1), \ldots, \nu(x_n))\}$ for each $n$-ary relation $R$ of $\mathcal{A}$.
	
	\item $\{(x_1, \ldots, x_n, y)\ :\ f(\nu(x_1), \ldots, \nu(x_n)) = \nu(y)\}$ for each function $f : A^n \rightarrow A$ of $\mathcal{A}$.
\end{itemize}
We say that a computably presentable countable structure $\mathcal{A}$ is 
\emph{relatively computably categorical} if whenever 
$\mathcal{A}^\#$ and $\mathcal{A}^+$ are computable presentations of $\mathcal{A}$, the join of their diagrams computes an isomorphism of $\mathcal{A}^\#$ onto $\mathcal{A}^+$. 

S. Goncharov gave a syntactic characterization of the relatively computable categorical countable structures \cite{Goncharov.1975}.   Clearly every relatively computably categorical structure is computably 
categorical.  S. Goncharov also constructed a computably categorical structure that is not relatively computably categorical \cite{Goncharov.1977}.  Numerous extensions of these results have been proven; see e.g. the survey \cite{Fokina.Harizanov.Melnikov.2014}.

The effect of structure on the separation of relative computable categoricity from computable categoricity has also been examined.  For example, a relatively computably categorical countable structure is computably categorical if it is either a linear order, a Boolean algebra, or an Abelian $p$-group \cite{Dzgoev.Goncharov.1980}, \cite{Remmel.1981.2}, \cite{Goncharov.1980.2}, \cite{Smith.1981}, \cite{Calvert.Cenzer.Harizanov.Morozov.2009}.

We now turn to the foundations of computable structure theory on analytic spaces.

\subsubsection{Computability on Banach spaces}

Our approach to computable structure theory on Banach spaces parallels the development of computable structure theory on metric spaces in \cite{Greenberg.Knight.Melnikov.Turetsky.2016}; see also \cite{Pour-El.Richards.1989}.  We first define what is meant by a computable presentation of a Banach space.  We then define for a computable presentation of a Banach space the associated computable vectors, sequences, c.e. open sets, and c.e. closed sets.  We then define the computable maps for computable presentations of Banach spaces.  After we summarize fundamental relationships between these notions, we define computable categoricity for Banach sapces.

Suppose $\mathcal{B}$ is a Banach space and $\mathcal{B}^\# = (\mathcal{B}, D)$ is a presentation of 
$\mathcal{B}$.  We say that $\mathcal{B}^\#$ is a \emph{computable presentation} of $\mathcal{B}$
if the norm is computable on the rational vectors of $\mathcal{B}^\#$; more formally if there is an algorithm 
that given any nonnegative integer $k$ and any finite sequence of scalars $\alpha_0, \ldots, \alpha_M \in \Q(i)$ computes a rational number $q$ so that $\left|\norm{\sum_j \alpha_j D(j)} - q\right| < 2^{-k}$.  The standard presentation $\C$ is a computable presentation as is the standard presentation of $L^p[0,1]$ when $p \geq 1$ is a computable real.  We say that $\mathcal{B}$ is \emph{computably presentable} if it has a computable presentation.

We note that if $\mathcal{B}^\#$ is a computable presentation of a Banach space $\mathcal{B}$, and if $S$ 
is a finite set, then $(\mathcal{B}^S)^\#$ (as defined in Section \ref{sec:classical}) is a computable presentation of $\mathcal{B}^S$.   

We now define the computable vectors and sequences of a computable presentation of a Banach space.  Fix a Banach space $\mathcal{B}$ and a computable presentation $\mathcal{B}^\#$ of $\mathcal{B}$.  A vector $v \in \mathcal{B}$ is a \emph{computable vector of $\mathcal{B}^\#$} if 
there is an algorithm that given any nonnegative integer $k$ computes a rational vector $u$ of $\mathcal{B}^\#$ so that $\norm{v - u} < 2^{-k}$; in other words, it is possible to compute arbitrarily good approximations of $v$.  If $v$ is a computable vector of $\mathcal{B}^\#$, then a code of such an algorithm will be referred to as an \emph{index} of $v$.  
A sequence $\{v_n\}_{n \in \N}$ of vectors in $\mathcal{B}$ is a \emph{computable sequence of $\mathcal{B}^\#$} if there is an algorithm that given any nonnegative integers $k,n$ as input computes a rational vector 
$u$ of $\mathcal{B}^\#$ so that $\norm{u - v_n} < 2^{-k}$; in other words, $v_n$ is computable uniformly in $n$.
If $\{v_n\}_{n \in \N}$ is a computable sequence of vectors of $\mathcal{B}^\#$, then a code of such an algorithm shall be referred to as an index of $\{v_n\}_{n \in \N}$.  

Suppose $L^p(\Omega)^\#$ is a computable presentation of $L^p(\Omega)$, and let $S$ be a set of nodes.
A map $\phi : S \rightarrow L^p(\Omega)$ is a \emph{computable map of $S$ into $L^p(\Omega)^\#$} if 
there is an algorithm that computes an index of $\phi(\nu)$ from $\nu$ if $\nu \in S$ and does not halt on any node that is not in $S$.  

We now define the c.e. open and closed subsets of a computable presentation $\mathcal{B}^\#$ of a Banach space $\mathcal{B}$.  An open subset $U$ of $\mathcal{B}$ is a \emph{c.e. open subset of $\mathcal{B}^\#$} if the set of all open rational balls of $\mathcal{B}^\#$ that are 
included in $U$ is c.e..  If $U$ is a c.e. open subset of $\mathcal{B}^\#$, then an index of $U$ is 
a code of a Turing machine that enumerates all open rational balls that are included in $U$.  
A closed subset $C$ of $\mathcal{B}$ is a \emph{c.e. closed subset of $\mathcal{B}^\#$} if the set of all open rational balls of $\mathcal{B}^\#$ that contain a point of $C$ is c.e..  If $C$ is a c.e. closed subset of $\mathcal{B}^\#$, then an \emph{index} of $C$ is a code of a Turing machine that enumerates all open rational balls that 
contain a point of $C$.

Now we define computable maps.  Suppose $\mathcal{B}_1^\#$ is a computable presentation of $\mathcal{B}_1$ and 
$\mathcal{B}_2^\#$ is a computable presentation of $\mathcal{B}_2$.  A map $T : \mathcal{B}_1 \rightarrow \mathcal{B}_2$ is a \emph{computable map of $\mathcal{B}_1^\#$ into $\mathcal{B}_2^\#$} if 
there is a computable function $P$ that maps rational balls of $\mathcal{B}_1^\#$ to rational balls of 
$\mathcal{B}_2^\#$ so that $T[B_1] \subseteq P(B_1)$ whenever $P(B_1)$ is defined and so that whenever $U$ is a neighborhood of $T(v)$, 
there is a rational ball $B_1$ of $\mathcal{B}_1^\#$ so that $v \in B_1$ and $P(B_1) \subseteq U$.  In other words, it is possible to compute arbitrarily good approximations of $T(v)$ from sufficiently good approximations of $v$.  An index of such a function $P$ will be referred to as an index of $T$.  
Suppose $\mathcal{B}_1^\# = (\mathcal{B}_1, R_1)$.  It is well-known that if $T$ is a bounded linear operator of $\mathcal{B}_1$ into $\mathcal{B}_2$, then $T$ is computable if and only if $\{T(R_1(n))\}_n$ is a computable sequence of $\mathcal{B}_2^\#$.  This principle holds uniformly if one is also provided with a bound on the operator $T$.  That is, from an upper bound on $\norm{T}$ and an index of $\{T(R_1(n))\}_n$ one can compute an index of $T$.  

Note that if $L^p(\Omega)^\#$ is a computable presentation of $L^p(\Omega)$, then 
$\sigma$ is a computable real-valued map from $(L^p(\Omega)^S)^\#$ into $\C$.

The following are `folklore' and follow easily from the definitions.

\begin{proposition}\label{prop:preimage.c.e.open}
Suppose $\mathcal{B}_1$, $\mathcal{B}_2$ are Banach spaces.  Let $\mathcal{B}_j^\#$ be a computable 
presentation of $\mathcal{B}_j$ for each $j$, and let $T$ be a computable map of $\mathcal{B}_1^\#$ into $\mathcal{B}_2^\#$.  Then, $T^{-1}[U]$ is a c.e. open subset of $\mathcal{B}_2^\#$ whenever $U$ is a c.e. open subset of $\mathcal{B}_2^\#$.  Furthermore, an index of $T^{-1}[U]$ can be computed from indices of $T$ and $U$.
\end{proposition}

\begin{proposition}\label{prop:bounding}
Suppose $\mathcal{B}$ is a Banach space, and let $\mathcal{B}^\#$ be a computable presentation of 
$\mathcal{B}$.   If $f$ is a computable real-valued function from $\mathcal{B}^\#$ into $\C$ with the property that $f(v) \geq d(v, f^{-1}[\{0\}])$ for all $v \in \mathcal{B}$, then, $f^{-1}[\{0\}]$ is c.e. closed.  Furthermore, an index of $f^{-1}[\{0\}]$ can be computed from an index of $f$.
\end{proposition}

\begin{proposition}\label{prop:comp.point}
Suppose $\mathcal{B}$ is a Banach space and $\mathcal{B}^\#$ is a computable presentation of 
$\mathcal{B}$.  Let $U$ be a c.e. open subset of $\mathcal{B}^\#$, and let $C$ be a c.e. closed subset 
of $\mathcal{B}^\#$ so that $C \cap U \neq \emptyset$.  Then, $C \cap U$ contains a computable vector of 
$\mathcal{B}^\#$.  Furthermore, an index of such a vector can be computed from indices of $U,C$.
\end{proposition}

\begin{proposition}\label{prop:eff.cauchy}
Suppose $\mathcal{B}$ is a Banach space and $\mathcal{B}^\#$ is a computable presentation of 
$\mathcal{B}$.  Let $\{v_n\}_{n \in \N}$ be a computable sequence of $\mathcal{B}^\#$ so that 
$\norm{v_n - v_{n+1}} < 2^{-n}$ for all $n \in \N$.  Then, 
$\lim_n v_n$ is a computable vector of $\mathcal{B}^\#$.  Furthermore, an index of $\lim_n v_n$
can be computed from an index of $\{v_n\}_{n \in \N}$.  
\end{proposition}

We now define a Banach space $\mathcal{B}$ to be \emph{computably categorical} if any two of its computable presentations are computably isometrically isomorphic; equivalently if $\mathcal{B}_1^\#$ is computably isomorphically isometric to $\mathcal{B}_2^\#$ whenever $\mathcal{B}_1^\#$, $\mathcal{B}_2^\#$ are computable presentations of Banach spaces that are isomorphically isometric to $\mathcal{B}$.

\subsubsection{Summary of prior work in analytic computable structure theory}\label{subsec:survey.prior}

The earliest work in analytic computable structure theory is implicit in the 1989 monograph of Pour-El and Richards \cite{Pour-El.Richards.1989}; namely, it is shown that $\ell^1$ is not computably categorical but that 
all separable Hilbert spaces are.  But, there was no more progress until 2013 when a number of results on metric spaces appeared.  In particular, Melnikov and Nies showed 
that computably presentable compact metric spaces are $\Delta_3^0$-categorical and that there is a computably presentable Polish space that is not $\Delta_2^0$-categorical \cite{Melnikov.Nies.2013}.  At the same time, Melnikov showed that the Cantor space, Urysohn space, and all separable Hilbert spaces are computably categorical (as metric spaces), but that (as a metric space) $C[0,1]$ is not \cite{Melnikov.2013}.  
Recently, Greenberg, Knight, Melnikov, and Turetsky announced an analog of Goncharov's syntactic characterization of relative computable categoricity for metric spaces \cite{Greenberg.Knight.Melnikov.Turetsky.2016}.

New results on Banach spaces began to appear in 2014.  First, Melnikov and Ng showed that $C[0,1]$ is not computably categorical \cite{Melnikov.Ng.2014}.  Then, McNicholl extended the work of Pour-El and Richards by showing that $\ell^p$ is computably categorical only when $p = 2$ and that $\ell^p$ is $\Delta_2^0$-categorical when $p$ is a computable real.  McNicholl also showed that $\ell^p_n$ is computably categorical when $p$ is a computable real and $n$ is a positive integer.   More recently, McNicholl and Stull have shown that whenever $(\ell^p)^\#$ is a computable presentation of $\ell^p$, there is a least powerful Turing degree that computes an isometric isomorphism of $\ell^p$ onto $(\ell^p)^\#$, and that these degrees are precisely the c.e. degrees \cite{McNicholl.Stull.2016}.

\section{Overview of the proof of Theorem \ref{thm:main}}\label{sec:overview}

As noted, every separable $L^2$ space is computably categorical since it is a Hilbert space.  So, we can confine ourselves to the case $p \neq 2$.  The three key steps to our proof of Theorem \ref{thm:main} are encapsulated in the following three theorems.

\begin{theorem}\label{thm:lifting.comp}
Let $L^p(\Omega_1)^\#$ be a computable presentation of $L^p(\Omega_1)$, and let 
$L^p(\Omega_2)^\#$ be a computable presentation of $L^p(\Omega_2)$.  
Suppose 
there is a computable disintegration of $L^p(\Omega_1)^\#$ that is computably isomorphic to a computable disintegration of $L^p(\Omega_2)^\#$.  
Then, there is a computable 
linear isometry of $L^p(\Omega_1)^\#$ onto $L^p(\Omega_2)^\#$.
\end{theorem}

\begin{theorem}\label{thm:comp.disint.Lp01}
Let $p$ be a computable real so that $p \geq 1$, and let $\Omega$ be a non-atomic separable measure space.   
Suppose $L^p(\Omega)^\#$ is a computable presentation of $L^p(\Omega)$, and suppose $\phi$ is a computable disintegration of $L^p(\Omega)^\#$ so that $\norm{\phi(\lambda)}_p = 1$.  Then, there is a computable disintegration of $L^p[0,1]$ that is computably isomorphic to $\phi$.
\end{theorem}

\begin{theorem}\label{thm:disint.comp}
Let $p \geq 1$ be a computable real so that $p \neq 2$.   
Suppose $\Omega$ is a separable nonzero measure space, and 
suppose $L^p(\Omega)^\#$ is a computable presentation of $L^p(\Omega)$.   Then, there is a computable disintegration of $L^p(\Omega)^\#$.
\end{theorem}

Theorem \ref{thm:main} follows immediately from Theorems \ref{thm:lifting.comp} through Theorem \ref{thm:disint.comp}.   Our proofs of each of these theorems are supported by a certain amount of classical material (that is, material that is devoid of computability content) which is developed in Section \ref{sec:classical}.  The transition from the classical realm to the computable is effected in Section \ref{sec:computable}.

\section{Classical world}\label{sec:classical}

We divide our work into three parts: isomorphism of disintegrations, extension of partial disintegrations, and approximation of separating antitone maps.  Subsection \ref{subsec:isomorphism} contains our results on isomorphism of disintegrations; this material provides the classical component of the proofs of Theorems \ref{thm:lifting.comp} and \ref{thm:comp.disint.Lp01}.  Subsection \ref{subsec:extension} contains our results on extensions of partial disintegrations, and our theorem on approximation of separating antitone maps appears in Subsection \ref{subsec:approx}; the results in these two subsections support our proof of Theorem \ref{thm:disint.comp}.   

\subsection{Isomorphism results}\label{subsec:isomorphism}

Our proof of Theorem \ref{thm:lifting.comp} is based on the idea that an isomorphism can be lifted to form a linear isometry.  We make this precise as follows.

\begin{definition}\label{def:lifts}
Suppose $\phi_1$, $\phi_2$ are disintegrations of $L^p(\Omega_1)$ and $L^p(\Omega_2)$ respectively, and suppose $f$ is an isomorphism of $\phi_1$ with $\phi_2$.  
We say that $T : L^p(\Omega_1) \rightarrow L^p(\Omega_2)$ \emph{lifts} $f$ if 
$T(\phi_1(\nu)) = \phi_2(f(\nu))$ for all $\nu \in \dom(\phi_1)$.   
\end{definition}

We show here that liftings of isomorphisms exist and are unique.  Namely, we prove the following.

\begin{theorem}\label{thm:extension.isomorphism}
Suppose $\Omega_1, \Omega_2$ 
are measure spaces and that $\phi_j$ is a disintegration of $L^p(\Omega_j)$ for each $j$.  
Suppose $f$ is an isomorphism of $\phi_1$ with $\phi_2$.  Then, there is a unique linear 
isometry of $L^p(\Omega_1)$ onto $L^p(\Omega_2)$ that lifts $f$.  
\end{theorem}

In Section \ref{sec:computable} we complete the proof of Theorem \ref{thm:lifting.comp} by showing that if $f$, $\phi_1$, $\phi_2$ are computable then the lifting of $f$ is computable.

The proof of Theorem \ref{thm:comp.disint.Lp01} is based on the following.

\begin{proposition}\label{prop:iso.int.disint}
Let $\Omega$ be a nonzero non-atomic measure space, and let $\phi$ be a disintegration of $L^p(\Omega)$ so that 
$\norm{\phi(\lambda)}_p = 1$.  Suppose 
$\psi$ is an interval-valued separating antitone map that is isomorphic to $\phi$,  
and suppose $\dom(\psi)$ is a tree.  Then, $\psi$ is a disintegration of $L^p[0,1]$.  
\end{proposition}

In Section \ref{sec:computable}, we complete the proof of Theorem \ref{thm:comp.disint.Lp01} by showing that when $\phi$ is computable there \emph{is} a computable interval-valued separating antitone map that is computably isomorphic to $\phi$.

We now proceed with the proofs of Theorem \ref{thm:extension.isomorphism} and Proposition \ref{prop:iso.int.disint}.

\begin{proof}[Proof of Theorem \ref{thm:extension.isomorphism}:]
Suppose $\Omega_j = (X_j, \mathcal{M}_j, \mu_j)$.  We first define a linear map $T$ on the linear span of $\ran(\phi_1)$.  In particular, we let 
\[
T(\sum_{\nu \in F} \alpha_\nu \phi_1(\nu)) = \sum_{\nu \in F} \alpha_\nu \phi_2(f(\nu))
\]
for every finite $F \subseteq \dom(\phi_1)$ and every corresponding family of scalars 
$\{\alpha_\nu\}_{\nu \in F}$.  

We first show that $T$ is well-defined.  Suppose 
\[
g = \sum_{\nu \in F_1} \alpha_\nu \phi_1(\nu) = \sum_{\nu \in F_2} \beta_\nu \phi_1(\nu).
\]
Without loss of generality, we assume $F_1 = F_2 = F$ where $F$ is a finite tree.   
We first make some observations.  Suppose $\phi : F \rightarrow L^p(\Omega)$ is a 
separating antitone map.  Let:
\begin{eqnarray*}
\nabla_\phi(\nu) & = & \phi(\nu) - \sum_{\nu' \in \nu^+ \cap F} \phi(\nu')\\
S_\phi(\nu) & = & \supp(\nabla_\phi(\nu))\\
\end{eqnarray*}
Note that $S_\phi(\nu) = \supp(\phi(\nu)) - \bigcup_{\nu' \in \nu^+ \cap F} \supp(\phi(\nu'))$ and that 
$\supp(\phi(\lambda)) = \bigcup_\nu S_\phi(\nu)$.  Also note that $\nabla_\phi(\nu)$ and $\nabla_\phi(\nu')$ are disjointly supported when $\nu \neq \nu'$.
We claim that if $\gamma_\nu \in \C$ for each $\nu \in F$, then 
\[
\sum_{\nu \in F} \gamma_\nu \phi(\nu) = \sum_{\nu \in F} \left( \sum_{\mu \subseteq \nu} \gamma_\mu \right) \nabla_\phi(\nu).
\]
For, when $\mu \subseteq \nu$, 
\begin{eqnarray*}
\nabla_\phi(\nu) & = & \phi(\nu)\cdot \chi_{S_\phi(\nu)} \\
& = & \phi(\mu) \cdot \chi_{\supp(\phi(\nu))} \chi_{S_\phi(\nu)}\\
& = & \phi(\mu) \chi_{S_\phi(\nu)}
\end{eqnarray*}
And, $\phi(\mu) \cdot \chi_{S_\phi(\nu)} = \mathbf{0}$ if $\mu \not \subseteq \nu$.  So, 
\begin{eqnarray*}
\sum_{\nu \in F} \gamma_\nu \phi(\nu) & = & \sum_{\nu \in F} \left( \sum_{\mu \in F} \gamma_\mu \phi(\mu) \right) \cdot \chi_{S_\phi(\nu)}\\
& = & \sum_{\nu \in F} \left( \sum_{\mu \subseteq \nu} \gamma_\mu \phi(\mu) \cdot \chi_{S_\phi(\nu)} \right) \\
& = & \sum_{\nu \in F} \left(\sum_{\mu \subseteq \nu} \gamma_\mu \right) \nabla_\phi(\nu).
\end{eqnarray*}
Thus, 
\[
g = \sum_{\nu \in F} \left(\sum_{\mu \subseteq \nu} \alpha_\mu\right) \nabla_{\phi_1}(\nu) = 
\sum_{\nu \in F} \left( \sum_{\mu \subseteq \nu} \beta_\mu \right) \nabla_{\phi_1}(\nu).
\]
Since nonzero disjointly supported vectors are linearly independent, it follows that 
\[
\sum_{\mu \subseteq \nu} \alpha_\mu = \sum_{\mu \subseteq \nu} \beta_\mu
\]
whenever $\nabla_{\phi_1}(\nu) \neq \mathbf{0}$.  

Let $\psi = \phi_2 \circ f$.  Since $f$ is an isomorphism, $\psi$ is a disintegration, and 
$\norm{\nabla_\psi(\nu)}_p = \norm{\nabla_{\phi_1}(\nu)}_p$ for all $\nu \in F$.  
Thus, 
\begin{eqnarray*}
\sum_{\nu \in F} \alpha_\nu \psi(\nu) & = & \sum_{\nu \in F} \left( \sum_{\mu \subseteq \nu} \alpha_\mu\right) \nabla_\psi(\nu)\\
& = & \sum_{\nu \in F} \left( \sum_{\mu \subseteq \nu} \beta_\mu\right) \nabla_\psi(\nu)\\
& = & \sum_{\nu \in F} \beta_\nu \psi(\nu).
\end{eqnarray*}
Thus, $T$ is well-defined.

We now also note that 
\begin{eqnarray*}
\norm{f}_p^p & = & \sum_{\nu \in F} \left| \sum_{\mu \subseteq \nu} \alpha_\mu \right|^p \norm{\nabla_{\phi_1}(\nu)}_p^p\\
& = & \sum_{\nu \in F} \left| \sum_{\mu \subseteq \nu} \alpha_\mu \right|^p \norm{\nabla_{\psi}(\nu)}_p^p\\
& = & \norm{T(f)}_p^p.
\end{eqnarray*}

It now follows that $T$ extends to a unique isometric linear map of $L^p(\Omega_1)$ into $L^p(\Omega_2)$; 
denote this map by $T$ as well.  
Since $\ran(\phi_2) \subseteq \ran(T)$, it follows that $T$ is surjective. 

Now, suppose $S$ is an isometric linear map of $L^p(\Omega_1)$ onto $L^p(\Omega_2)$ so that 
$S(\phi_1(\nu)) = \phi_2(f(\nu))$ for all $\nu \in \dom(\phi_1)$.  So, $S(\phi_1(\nu)) = T(\phi_1(\nu))$ for all $\nu \in \dom(\phi_1))$. 
That is, $S(f) = T(f)$ whenever $f \in \ran(\phi_1)$.  Since the linear span of $\ran(\phi_1)$ is dense
in $L^p(\Omega_1)$, it follows that $S = T$.  
\end{proof}

To prove Proposition \ref{prop:iso.int.disint}, we will need the following lemma.

\begin{lemma}\label{lm:descending}
Suppose $\{f_n\}_n$ is a sequence of vectors in $L^p(\Omega)$ so that $f_{n+1} \preceq f_n$ for all $n$.  Then, $\{f_n\}_n$ converges in the $L^p$-norm.
\end{lemma}

\begin{proof}
Since $f_{n+1} \preceq f_n$, it follows that $\norm{f_{n+1}}_p \leq \norm{f_n}_p$ and moreover that 
$\norm{f_n - f_m}_p = \norm{f_m}_p - \norm{f_n}_p$ whenever $n \geq m$.  So, on the one hand $\lim_n \norm{f_n}_p$ exists.  On the other hand, this implies that $\{f_n\}_n$ is a Cauchy sequence.  
Thus, $\{f_n\}_n$ converges in the $L^p$-norm.
\end{proof}

\begin{proof}[Proof of Proposition \ref{prop:iso.int.disint}:]
First, we claim that if $\epsilon > 0$, then there exists $n$ so that $\max\{\norm{\psi(\nu)}_p\ :\ |\nu| = n\} < \epsilon$.  For, suppose otherwise.  Let $S = \{\nu \in \dom(\psi)\ :\ \norm{\psi(\nu)}_p \geq \epsilon\}$.  
Thus, since $\psi$ is an antitone map, $S$ is a tree.  
Since $\psi$ is interval-valued, $S$ is a finitely branching tree.  Let $\beta$ be an infinite branch of $S$; that is, $\beta$ is a function from $\N$ into $S$ so that 
$\beta(n+1) \supset \beta(n)$ for all $n \in \N$.     
Let $f$ be an isomorphism of $\psi$ onto $\phi$.  Then, by Lemma \ref{lm:descending}, $\lim_n f(\beta(n))$ exists in the $L^p$-norm; let 
$h$ denote this limit.  Then, $\norm{h}_p \geq \epsilon$ and 
\[
\langle \ran(\phi) \rangle \subseteq \langle h \rangle \oplus \{g \in L^p[0,1]\ :\ \supp(g) \cap \supp(h) = \emptyset\}.
\]  
Since $\Omega$ is non-atomic, by Theorem \ref{thm:atomic} and Proposition \ref{prop:abs.cont}, there is 
a measurable set $A$ so that $\norm{h \cdot \chi_A}_p = \epsilon / 2$.
Therefore, $h \cdot \chi_A \not \in \langle \ran(\phi) \rangle$; a contradiction.

Now, to show that $\psi$ is a disintegration, it suffices to show that $\langle \ran(\psi) \rangle = L^p[0,1]$.  To this end, it suffices to show that $\chi_I \in \langle \ran(\psi) \rangle$ whenever $I$ is a subinterval of $[0,1]$.  Now suppose $[a,b] \subseteq [0,1]$.  Since $\psi$ is interval-valued, for each $\nu$, there is an interval $I(\nu) \subseteq [0,1]$ so that $\psi(\nu) = \chi_{I(\nu)}$.  Choose $\epsilon > 0$ and $n$ so that 
$\norm{\psi(\nu)}_p < \epsilon$ whenever $\nu \in \dom(\phi)$ and $|\nu| = n$.  
Since $\phi$ is a disintegration, and since $\norm{\phi(\lambda)}_p = 1$, it follows that 
$\bigcup_{|\nu| = n} I(\nu) = [0,1]$.  
Let $F = \{\nu \in \dom(\phi)\ :\ |\nu| = n \wedge\ I(\nu) \cap [a,b] \neq \emptyset\}$.  
So, $[a,b] \subseteq \bigcup_{\nu \in F} I(\nu)$.  
Thus, $\mu(\bigcup_{\nu \in F} I(\nu) - [a,b]) < 2\epsilon$, and therefore $\chi_{[a,b]} \in \langle \ran(\psi) \rangle$.
Hence, $\langle \ran(\psi) \rangle = L^p[0,1]$.
\end{proof}

\subsection{Extending partial disintegrations}\label{subsec:extension}

Our goal in this subsection is to prove the following which will support our proof of Theorem \ref{thm:disint.comp}. 

\begin{theorem}\label{thm:extend.partial.disint}
Suppose $\Omega$ is a separable measure space and $1 \leq p < \infty$.  Suppose $\phi$ is a partial disintegration of $L^p(\Omega)$.  Then, for every finite subset $F$ of 
$L^p(\Omega)$ and every nonnegative integer $k$, $\phi$ extends to a partial disintegration $\psi$ so that 
$d(f, \langle \ran(\psi) \rangle) < 2^{-k}$ for every $f \in F$.
\end{theorem}

We divide the majority of the proof of Theorem \ref{thm:extend.partial.disint} into a sequence of lemmas as follows.

\begin{lemma}\label{lm:approx.reciprocal}
Let $\Omega$ be a measure space and suppose $1 \leq p < \infty$.  Let $f \in L^p(\Omega)$ be supported on a set of finite measure.  Then, for every $\epsilon > 0$, there is a simple function $s$ so that 
$\supp(s) \subseteq \supp(f)$ and $\norm{s \cdot f - \chi_{\supp(f)}}_p < \epsilon$.
\end{lemma}

\begin{proof}
Suppose $\Omega = (X, \mathcal{M}, \mu)$.  Let $A = \supp(f)$.  Without loss of generality, suppose $\norm{f}_p > 0$.  Let $\epsilon > 0$.  For each nonnegative integer $k$ let 
\[
A_k = \{t \in X\ :\ |f(t)| > 2^{-k}\}.
\]
Since $\mu(A) < \infty$, $\lim_k \mu(A - A_k) = 0$.  Choose $k$ so that $\mu(A - A_k) < \epsilon /2$.  Set $g = (1/f) \cdot \chi_{A_k}$.  Thus, $g \in L^\infty(\Omega)$.  So, there is a simple function $s$ so that $\supp(s) \supseteq A_k$ and $\norm{s - g}_\infty < \frac{1}{2} \epsilon \norm{f}_p$.  Then, 
\begin{eqnarray*}
\norm{s \cdot f - \chi_{A_k}}_p^p & = & \norm{(s - f) \cdot f}_p^p \\
& = & \norm{ |s - g|^p |f|^p }_1\\
& \leq & \norm{ |s - g|^p}_\infty \norm{|f|^p}_1\\
& = & \norm{s - g}_\infty^p \norm{f}_p^p\\
& < & 2^{-p} \epsilon^p
\end{eqnarray*}
So, 
\begin{eqnarray*}
\norm{s \cdot f - \chi_A}_p & \leq & \norm{s \cdot f - \chi_{A_k}}_p + \norm{\chi_{A_k} - \chi_A}_p\\
& < & \epsilon
\end{eqnarray*}
\end{proof}

\begin{lemma}\label{lm:density.vectors}
Suppose $\Omega$ is a measure space and $1 \leq p < \infty$.  Suppose 
$\mathcal{D} \subseteq L^p(\Omega)$ is a simple lower semilattice with the property that the upper semilattice generated by the supports of the vectors in $\mathcal{D}$ is dense in $\Omega$.  Then, the linear span of $\mathcal{D}$ is dense in $L^p(\Omega)$.
\end{lemma}

\begin{proof}
Suppose $\Omega = (X, \mathcal{M}, \mu)$.  
It suffices to show that if $\mu(A) < \infty$, then $\chi_A$ is a limit point of the linear span of $\mathcal{D}$ in the $L^p$-norm.  So, suppose $\mu(A) < \infty$.  Choose $f_0, \ldots, f_n \in \mathcal{D}$ so that 
$\mu(A \triangle \bigcup_{j = 0}^n \supp(f_j)) < \epsilon / 3$.  Since $\mathcal{D}$ is simple, we can assume $f_0, f_1, \ldots, f_n$ are disjointly supported.  
Thus, 
\[
\norm{\chi_A - \sum_{j = 0}^n \chi_{\supp(f_j)}}_p < \epsilon / 3.
\]
Set $f = \sum_{j = 0}^n f_j$, and set $B = \bigcup_{j \leq n} \supp(f_j)$.  
Thus, $B = \supp(f)$.  By Lemma \ref{lm:approx.reciprocal}, there is a simple function $s$ so that 
$\norm{sf - \chi_B}_p < \epsilon/3$.  Hence, $\norm{\chi_A - sf}_p < 2\epsilon/3$.  

Let $s = \sum_{j = 0}^k \alpha_j \chi_{A_j}$ where $A_0, \ldots, A_k$ are pairwise disjoint measurable 
subsets of $X$ and $\alpha_0, \ldots, \alpha_k$ are nonzero.  Thus, $\mu(A_j) < \infty$ and 
\[
sf = \sum_{j = 0}^k \alpha_j f \chi_{A_j}.
\]
Set $M = \max\{|\alpha_0|, \ldots, |\alpha_k|\}$.  Choose $\delta > 0$ so that 
\[
\int_E |f|^p d\mu < \left(\frac{\epsilon}{3}\right)^p \frac{1}{(k+1)M}
\]
whenever $E$ is a measurable subset of $X$ so that $\mu(E) < \delta$.  For each $j$, there exist 
$g_{j,0}, \ldots, g_{j, m_j} \in \mathcal{D}$ so that $\mu(A_j \triangle \bigcup_s \supp(g_{j,s}))< \delta$.  
Set $B_{j,s} = \supp(g_{j,s})$ and let $H_j = \bigcup_s B_{j,s}$.  Thus, 
\begin{eqnarray*}
\norm{sf - \sum_j \alpha_jf\chi_{A_j}}_p^p & \leq & \sum_j |\alpha_j| \int_X |f|^p| |\chi_{A_j} - \chi_{H_j}|\ d\mu\\
& = & \sum_j |\alpha_j| \int_{A_j \triangle H_j} |f|^p\ d\mu\\
& \leq & M(k+1) \left(\frac{\epsilon}{3}\right)^p \frac{1}{M(k+1)} = \left(\frac{\epsilon}{3}\right)^p.
\end{eqnarray*}
Thus, $\norm{\chi_A - \sum_j \alpha_j f \chi_{H_j}}_p < \epsilon$.

Now, note that 
\[
f_t \chi_{B_{j,s}} = \left\{
\begin{array}{cc}
0 & \mbox{if $\mu(B_{j,s} \cap \supp(f_t)) = 0$}\\
f_t & \mbox{if $f_t \preceq g_{j,s}$}\\
g_{j,s} & \mbox{if $g_{j,s} \preceq f_t$}
\end{array}
\right.
\]
It follows that $\sum_j \alpha_j f \chi_{H_j}$ belongs to the linear span of $\mathcal{D}$.
\end{proof}

\begin{lemma}\label{lm:semilattice.adjoin}
Suppose $\Omega$ is a measure space and $\mathcal{D}$ is a finite simple lower semilattice
of measurable sets.  Then, for every measurable set $A$ that does not belong to the upper semilattice generated by $\mathcal{D}$, $\mathcal{D}$ properly extends to a finite simple lower semilattice $\mathcal{D}'$ of measurable sets so that $A$ belongs to the upper semilattice generated by $\mathcal{D}'$.  
\end{lemma}

\begin{proof}
When $Y \in \mathcal{D}$, define the \emph{remnant} of $Y$ to be 
\[
Y - \bigcup\{Z\ :\ Z \in \mathcal{D}\ \wedge\ Z \subset Y\}.
\]
Let $\mathcal{R}$ denote the set of all remnants of sets in $\mathcal{D}$.  Note that any two distinct sets in 
$\mathcal{R}$ are disjoint.  Let:
\begin{eqnarray*}
\mathcal{R}' & = & \{R \cap A\ :\ R \in \mathcal{R}\}\\
S_A & = & A - \bigcup \mathcal{D}\\
\mathcal{D}' & = & \mathcal{D} \cup \mathcal{R}' \cup \{S_A\}.
\end{eqnarray*}
We claim that $\mathcal{D}'$ is a simple lower semilattice.  For, suppose $X_1, X_2 \in \mathcal{D}'$ are incomparable.  We can suppose one of $X_1$, $X_2$ does not belong to $\mathcal{D}$.  We can also assume one of $X_1$, $X_2$ does not belong to $\mathcal{R}'$.  If $X_1$ or $X_2$ is $S_A$, then $X_1 \cap X_2 = \emptyset$.  So, we can assume $X_1 \in \mathcal{D}$ and $X_2 \in \mathcal{R}'$.  Thus, there exists 
a remnant $R$ of a set $Y \in \mathcal{D}$ so that $X_2 = R \cap A$.  Thus, $R \subseteq Y$.  So, 
$Y \not \subseteq X_1$.  If $X_1 \cap Y$ is null, then so is $X_1 \cap X_2$.  So, suppose $X_1 \subset X$.  
Then, $R \cap X_1 = \emptyset$, so $X_2 \cap X_1 = \emptyset$.

We now note that $\bigcup \mathcal{R} = \bigcup \mathcal{D}$.  Thus, $A = S_A \cup \bigcup \mathcal{R}'$, and so $A$ belongs to the upper semilattice generated by $\mathcal{D}'$.  Thus, $\mathcal{D} \subset \mathcal{D}'$.

We now show that $\mathcal{D}'$ properly extends $\mathcal{D}$.  Suppose $B \in \mathcal{D}' - \mathcal{D}$ and suppose $C$ is a nonzero set in $\mathcal{D}$.   If $B = S_A$, then 
$B \cap C = \emptyset$ and so $B \not \supseteq C$.  Suppose $B = R \cap A$ where $R$ is the remnant of $Y \in \mathcal{D}$.  By way of contradiction, suppose $B \supset C$.  Then, $Y \supset C$, and so $R \cap C = \emptyset$ which is impossible since $C$ is nonempty.  Thus, $\mathcal{D}'$ properly extends $\mathcal{D}$.
\end{proof}

\begin{lemma}\label{lm:extend.partial.disintegration}
Suppose $\phi$ is a partial disintegration of $L^p(\Omega)$, and suppose $\mathcal{D} \subseteq L^p(\Omega)$ is a finite simple lower semilattice of  vectors in $L^p(\Omega)$ that properly extends $\ran(\phi)$.  Then, 
$\phi$ extends to a partial disintegration $\psi$ of $L^p(\Omega)$ with range $\mathcal{D} - \{\mathbf{0}\}$.
\end{lemma}

\begin{proof}
Let $S = \dom(\phi)$.  By induction, we can assume $\#(\mathcal{D} - \ran\phi)=1$.  Suppose $f$ is the unique element of $\mathcal{D}- \ran(\phi)$. Since $g\not\preceq f$ for all $g\in\ran\phi$, precisely two cases arise. The first is $f\not\preceq g$ for all $g\in \ran\phi$. For this case, we let 
\[
t=\max\{z\in\N: (z) \in S\},
\]
set $S^\prime=S\cup\{(t+1)\}$, and define $\psi:S^\prime \to L^p(\Omega)$ by 
\[
\psi(\nu) = \left\{\begin{array}{cc}
			f & \nu = (t+1)\\
			\phi(\nu) & \nu \neq (t+1).
			\end{array}
			\right.
\]
By choice of $t$ and the incomparability of $(t+1)$ with each element of $S$, $\psi$ is an injective antitone map. That $\psi$ is separating follows from the incomparability of $f$ with any element of $\ran\phi$. Furthermore, $S^\prime \cup \{\emptyset\}$ is a finite subtree of $\N^*$, so $\psi$ is a partial disintegration onto $\mathcal{D}$ that extends $\phi$. 
	
The other case is that there exists $g\in \ran(\phi)$ so that $f\preceq g$.  Since $S$ is finite, there is a $\preceq$-minimal vector $g \in \ran(\phi)$ so that $f \preceq g$.  Since $\ran(\phi)$ is simple, and since $f$ is nonzero, $g$ is unique.  
Note that $f$ is incomparable with every element $h$ of $\ran(\phi)$ so that $g \not \preceq h$.  
We let: 
\begin{eqnarray*}
t & = & \max\{z\in\N\ :\ \phi^{-1}(g)^\frown(z) \in S\} \\
S^\prime & = & S\cup\{\phi^{-1}(g)^\frown(t+1)\} \\
\end{eqnarray*}
For all $\nu \in S'$, let 
\[
\psi(\nu) = \left\{\begin{array}{cc}
			f & \nu \in S' - S\\
			\phi(\nu) & \nu \in S
			\end{array}
			\right.
\]
By our choice of $g$ and $t$, $\psi$ is an injective antitone map. That $\psi$ is separating follows from the incomparability of $\phi^{-1}(g)^\frown(t+1)$ with every element $\mu$ of $S$ so that $\mu \not \subseteq \phi^{-1}(g)$.  The set $S^\prime$ is also a finite orchard, so $\psi$ is a partial disintegration onto $\mathcal{D}$ which extends $\phi$.
\end{proof}

\begin{proof}[Proof of Theorem \ref{thm:extend.partial.disint}:]
Let $\mathcal{R}  = \{R_0, R_1, \ldots\}$ be a countable dense set of measurable sets.  We build a set $\mathcal{D} \supseteq \ran(\phi)$ that satisfies 
the hypotheses of Lemma \ref{lm:density.vectors}.  To ensure this, we ensure that each set in $\mathcal{R}$ belongs to the upper semilattice generated by the supports of the vectors in $\mathcal{D}$.  
We construct $\mathcal{D}$ by defining a sequence $\mathcal{D}_0 \subseteq \mathcal{D}_1 \subseteq \ldots$ 
and setting $\mathcal{D} = \bigcup_n \mathcal{D}_n$.  To begin, set $\mathcal{D}_0 = \ran(\phi) \cup \{\mathbf{0}\}$.  

Let $n \in \N$, and suppose $\mathcal{D}_n$ has been defined.  Let $\mathcal{F} = \{\supp(f)\ :\ f \in \mathcal{D}_n\}$.   By way of induction, suppose $\mathcal{D}_n$ is a simple lower semilattice.  Thus, $\mathcal{F}$ is a simple lower semilattice of measurable sets.  By Lemma \ref{lm:semilattice.adjoin}, 
there is a simple lower semilattice $\mathcal{F}'$ so that $R_n$ belongs to the upper semilattice generated by 
$\mathcal{F}'$.  Let $h_1 = \bigvee \mathcal{D}_n$.  Let $h_2 = \chi_{R_n - \supp(h_1)}$.  Let 
$\mathcal{D}_{n+1} = \{(h_1 + h_2) \cdot \chi_S\ :\ S \in \mathcal{F}'\}$.  Thus, since $\mathcal{F}'$ is a simple lower semilattice, $\mathcal{D}_{n+1}$ is a simple
lower semilattice  under $\preceq$.  We claim that $\mathcal{D}_n \subseteq \mathcal{D}_{n+1}$.  
For, let $f \in \mathcal{D}_n$.  Thus, $S := \supp(f) \in \mathcal{F}$.  
So, $(h_1 + h_2)\cdot \chi_S \in \mathcal{D}_{n+1}$.  But, $(h_1 + h_2) \cdot \chi_S = h_1 \cdot \chi_S = f$.  

So, it follows from Lemma \ref{lm:density.vectors} that the linear span of $\mathcal{D}$ is dense in $L^p(\Omega)$.  So, there exists 
a finite $S\subseteq \mathcal{D} - \{\mathbf{0}\}$ so that $d(f, \langle S\rangle) < 2^{-k}$.  
We can assume $\ran(\phi) \subseteq S$.  We can now apply Lemma \ref{lm:extend.partial.disintegration}.
\end{proof}

\subsection{Approximating separating antitone maps}\label{subsec:approx}

We show that the $\sigma$ functional defined in Section \ref{sec:background} can be used to estimate distance to the nearest separating antitone map.

\begin{theorem}\label{thm:sigma.estimate}
Suppose $\Omega$ is a measure space and $p$ is a real so that 
$p \geq 1$ and $p \neq 2$.  Suppose $\phi : S \rightarrow L^p(\Omega)$ is a partial disintegration of $L^p(\Omega)$, and $\psi : S' \rightarrow L^p(\Omega)$ where $S' \supseteq S$ is a finite orchard so that each $\nu \in S' - S$ is a descendant of a node in $S$.  Then, there is a separating antitone map $\psi' : S' \rightarrow L^p(\Omega)$ so that 
\begin{equation}
\norm{\psi' - \psi}_{S'}^p  \leq \norm{\phi - \psi |_S}^p_S + 2^p\sigma(\phi \cup (\psi|_{S' -  S})). \label{eqn:sigma.estimate}
\end{equation}
\end{theorem}

\begin{proof}
Let $\Delta =  S' - S$.  Let $\psi_0 = \phi \cup \psi|_\Delta$. 
Let
\[
\hat{\sigma}(\psi_0)=\sum_{\nu|\nu^\prime}\min\{|\psi_0(v)|^p,|\psi_0(\nu^\prime)|^p\}
+\sum_{\nu^\prime\supset \nu}\min\{|\psi_0(\nu^\prime)-\psi_0(\nu)|^p,|\psi_0(\nu^\prime)|^p\} 
\]
where $\nu,\nu^\prime$ range over $S'$. 

When $\nu \in \Delta$, define the \emph{nullifiable set of $\nu$} to be the set of all $t \in X$ so that 
$|\psi(\mu)(t)|^p \leq \hat{\sigma}(\psi_0)(t)$ for some $\mu \subseteq \nu$ that belongs to $\Delta$.  
Denote the nullifiable set of $\nu$ by $N_\nu$.  When $\nu \in \Delta$, define the \emph{source node of $\nu$} 
to be the maximal $\mu \subseteq \nu$ so that $\mu \in S$.  Thus, each $\nu \in \Delta$ has a source node.  

Let $\nu \in S'$.  If $\nu \in S$, then define $\psi'(\nu)$ to be $\phi(\nu)$.  If $\nu \in \Delta$, then define 
$\psi'(\nu)$ to be $\phi(\mu)\cdot(1 - \chi_{N_\nu})$ where $\mu$ is the source node of $\nu$.  

Note that $N_\nu \subseteq N_{\nu'}$ if $\nu, \nu' \in \Delta$ and $\nu \subseteq \nu'$.  Thus, 
$\psi'$ is antitone.  

Suppose $\nu, \nu' \in \Delta$ are incomparable.  Suppose $t \not \in N_\nu$.  Then $|\psi(\nu)(t)|^p > \hat{\sigma}(\psi_0)(t)$.  So, $|\psi(\nu)(t)|^p > \hat{\sigma}(\psi_0)(t)$.  
Thus, 
\[
|\psi(\nu')(t)|^p = \min\{|\psi(\nu)(t)|^p, |\psi(\nu')(t)|^p\} \leq \hat{\sigma}(\psi_0)(t).
\]  
Hence, 
$t \in N_{\nu'}$.  Thus, $1 - \chi_{N_\nu}$ and $1 - \chi_{N_{\nu'}}$ are disjointly supported.  

So, suppose $\nu, \nu' \in S'$ are incomparable.  If either $\nu, \nu' \in S$ or if $\nu, \nu' \in \Delta$, then 
$\psi'(\nu)$ and $\psi'(\nu')$ are incomparable.  Suppose $\nu \in S$ and $\nu' \in \Delta$.  Let $\mu$
denote the source node of $\nu'$.  Then, $\mu \not \subset \nu$ and $\nu \not \subseteq \mu$. 
Thus, $\mu$, $\nu$ are incomparable and so $\psi'(\mu)$ and $\psi'(\nu)$ are disjointly supported.  
Thus, $\psi'(\nu')$ and $\psi'(\nu)$ are disjointly supported.

Now, note that 
\[
\norm{\psi - \psi'}_{S'}^p \leq \norm{\phi - \psi|_S}_S^p + \norm{(\psi - \psi')|_\Delta}_\Delta^p.
\]
Suppose $\nu \in \Delta$ and $t \in X$.  We claim that 
$|\psi(\nu)(t) - \psi(\nu')(t)|^p \leq 2^p \hat{\sigma}(\psi_0)(t)$.  For, suppose $t \not \in N_\nu$.  
Then, $\psi'(\nu)(t) = \phi(\mu)(t)$ where $\mu$ is the source node of $\nu$.  Also, 
$|\psi(\nu)(t)|^p > \hat{\sigma}(\psi_0)(t)$.  So, 
$|\psi(\nu)(t)|^p > \min\{| \phi(\mu)(t) - \psi(\nu)(t)|^p, |\psi(\nu)(t)|^p\}$. 
Thus, 
\[
|\phi(\mu)(t) - \psi(\nu)(t)|^p = \min\{| \phi(\mu)(t) - \psi(\nu)(t)|^p, |\psi(\nu)(t)|^p\} \leq \hat{\sigma}(\psi_0)(t) \leq 2^p \hat{\sigma}(\psi_0)(t).
\]
Suppose $t \in N_\nu$.  Then, $\psi'(\nu)(t) = 0$.  There exists $\mu' \subseteq \nu$ so that 
$|\psi'(\mu)(t)|^p \leq \hat{\sigma}(\psi_0)(t)$.  Without loss of generality, suppose 
$|\psi(\nu)(t)|^p > \hat{\sigma}(\psi_0)(t)$. So, $|\psi(\nu)(t)|^p > \min\{|\psi(\mu')(t) - \psi(\nu)(t)|^p, |\psi(\nu)(t)|^p\}$.  Therefore, $|\psi(\mu')(t) - \psi(\nu)(t)|^p \leq \hat{\sigma}(\psi_0)(t)$.  So, 
$|\psi(\nu)(t)|^p \leq 2^p\hat{\sigma}(\psi_0)(t)$ since $|a + b|^p \leq 2^{p-1}(|a|^p + |b|^p)$.  
\end{proof}

\section{Computable world}\label{sec:computable}

We now have all the pieces in place to prove Theorems \ref{thm:lifting.comp} and \ref{thm:comp.disint.Lp01}.

\begin{proof}[Proof of Theorem \ref{thm:lifting.comp}:]
Suppose $\phi_1$ is a computable disintegration of $L^p(\Omega_1)^\#$ and that 
$\phi_2$ is a computable disintegration of $L^p(\Omega_2)^\#$ that is computably isomorphic to $\phi_1$.  

The domain of $\phi_j$ is c.e., so there is a computable surjection $G_j'$ of $\N$ onto $\dom(\phi_j)$.  
Let $G_j = \phi_j \circ G_j'$.  Thus, $G_j$ is a structure on $L^p(\Omega_j)$.  So, let 
$L^p(\Omega_j)^+ = (L^p(\Omega_j), \phi_j)$.   Since $\phi_j$ is a computable disintegration of $L^p(\Omega_j)^\#$, it follows that $L^p(\Omega_j)^+$ is a computable presentation of $L^p(\Omega_j)$ and that the identity map is a computable map of $L^p(\Omega_j)^+$ onto $L^p(\Omega_j)^\#$.  

Let $f$ be a computable isomorphism of $\phi_1$ with $\phi_2$.  Thus, by Theorem \ref{thm:extension.isomorphism}, there is a unique linear isometric map of $L^p(\Omega_1)$ onto $L^p(\Omega_2)$ that lifts $f$; denote this map by $T$.  
Since $T$ lifts $f$, it follows that $\{T(G_1(n))\}_n$ is a computable sequence of $L^p(\Omega_2)^+$.  
Since $T$ is bounded, it follows that 
$T$ is a computable map of $L^p(\Omega_1)^+$ onto $L^p(\Omega_2)^+$.  
Thus, there is a computable linear isometry of $L^p(\Omega_1)^\#$ onto $L^p(\Omega_2)^\#$.  Since $T$ is an isometry, $\norm{T} = 1$.  Thus, the conclusion holds uniformly. 
\end{proof}

\begin{proof}[Proof of Theorem \ref{thm:comp.disint.Lp01}:]
Without loss of generality, suppose $\norm{\phi(\lambda)}_p = 1$.  
Let $S = \dom(\phi)$.  Every c.e. tree is computably isomorphic to a computable tree.  So, without loss of generality, we assume $S$ is computable.

Set $I(\lambda) = [0,1]$.  
Suppose $\nu \in S$, and let $\nu_0 <_{lex} \nu_1 <_{lex} \ldots$ be the children of $\nu$ in $S$.  
For each $n$, set 
\[
I(\nu_n) = \left[ \sum_{j < n} \norm{\psi(\nu_j)}_p^p, \sum_{j \leq n} \norm{\psi(\nu_j)}_p^p \right]. 
\]
Since $\phi(\nu) \preceq \phi(\lambda)$ for all $\nu \in S$, it follows that $I(\nu) \subseteq [0,1]$ for all $\nu$.  
Set $\psi(\nu) = \chi_{I(\nu)}$.  It follows that $\psi$ is a separating antitone map.  
It also follows that $\psi$ is computable, and that the identity map gives a computable isomorphism of $\psi$ with $\phi$.  Thus, by Proposition \ref{prop:iso.int.disint}, $\psi$ is a disintegration.  
\end{proof}

To prove Theorem \ref{thm:disint.comp}, we augment the classical material developed so far with the following three lemmas.  The third lemma requires the notion of a success index which we define now.

\begin{definition}\label{def:success.index}
Suppose $L^p(\Omega)^\# = (L^p(\Omega), R)$ is a presentation.  
Let $S$ be a finite orchard, and let $\psi : S \rightarrow L^p(\Omega)^\#$.  The \emph{success index of $\psi$} is the largest integer $N$ so that 
$d(R(j), \langle \ran(\psi) \rangle) < 2^{-N}$ whenever $0 \leq j < N$ and 
\[
\norm{ \psi(\nu) - \sum_{\nu' \in \nu^+ \cap S} \psi(\nu')}_p < 2^{-N}
\]
whenever $\nu$ is a nonterminal node of $S$. 
\end{definition}

The success index of an antitone separating map can be viewed as a measure of how close it is to being a 
disintegration. Antitone separating maps with larger success indices are closer to being disintegrations.

\begin{lemma}\label{lm:c.e.open}
Suppose $\mathcal{B}$ is a Banach space and let $\mathcal{B}^\#$ be a computable presentation of $\mathcal{B}$.  Let $S$ be a finite set of nodes.
\begin{enumerate}
	\item The set of all injective maps in $\mathcal{B}^S$ is a c.e. open subset of $\mathcal{B}^\#$;  furthermore, an index of this set can be computed from $S$.  \label{lm:c.e.open::itm:injective} 
	
	\item Suppose $v_0, \ldots, v_k$ are computable vectors of $\mathcal{B}^\#$ and $N \in \N$.  
	Then, the set of all $\psi \in \mathcal{B}^S$ so that 
	$d(v_j, \langle \ran(\psi) \rangle) < 2^{-N}$ whenever $0 \leq j \leq k$ is a c.e. open subset of 
	$\mathcal{B}^\#$; furthermore an index of this set can be computed from $N, S$ and indices of $v_0, \ldots, v_k$.  \label{lm:c.e.open::itm:distance}
	
	\item Suppose $N \in \N$ and $S$ is an orchard.  Then, the set of all $\psi \in \mathcal{B}_S$ so that 
\[
\norm{\psi(\nu) - \sum_{\mu \in \nu^+ \cap S} \psi(\mu) }_p < 2^{-N}
\]
for every nonterminal node $\nu$ of $S$ is a c.e. open subset of $\mathcal{B}^\#$; furthermore, and index of this set can be computed from $N$ and $S$.
 \label{lm:c.e.open::itm:summative}
\end{enumerate}
\end{lemma}

\begin{proof}
We will repeatedly use the following well-known fact: if $U,V$ are c.e. open subsets of $(\mathcal{B}^S)^\#$, then $U \cap V$ is a c.e. open subset of $(\mathcal{B}^S)^\#$ and an index of $U \cap V$ can be computed from indices of $U,V$.\\

(\ref{lm:c.e.open::itm:injective}): Let $\mathcal{S}_1$ denote the set of all injective maps in $\mathcal{B}^S$.  
Suppose $\nu, \nu' \in S$ are distinct.  Define $G_{\nu, \nu'} : \mathcal{B}^S \rightarrow \mathcal{B}$ by 
\[
G_{\nu, \nu'}(\psi) = \psi(\nu) - \psi(\nu').
\]
Then, $G_{\nu, \nu'}$ is a computable map of $(\mathcal{B}^S)^\#$ into $\mathcal{B}^\#$; furthermore an 
index of $G_{\nu, \nu'}$ can be computed from $S$, $\nu$, and $\nu'$.  
The set of nonzero vectors in $\mathcal{B}$ is a c.e. open subset of $\mathcal{B}^\#$.  
So, by Proposition \ref{prop:preimage.c.e.open}, $U_{\nu, \nu'} := G_{\nu, \nu'}^{-1}[\mathcal{B} - \{\mathbf{0}\}]$ is a c.e. open subset of $(\mathcal{B}^S)^\#$; furthermore an index of $U_{\nu, \nu'}$ can be computed from 
$\nu$, $\nu'$, and $S$.  
Since $\mathcal{S}_1 = \bigcap_{\nu, \nu'} U_{\nu, \nu'}$, it follows that 
$\mathcal{S}_1$ is a c.e. open subset of $(\mathcal{B}^S)^\#$ and that an index of $\mathcal{S}_1$
can be computed from $S$.\\

(\ref{lm:c.e.open::itm:distance}): 
By considering intersections, it suffices to consider the case where $k = 0$.  
Let $\mathcal{S}_2$ denote the set of all $\psi \in \mathcal{B}^S$ so that $d(v_0, \langle \ran(\psi) \rangle) < 2^{-N}$.  
Observe that $\psi \in \mathcal{S}_2$ if and only if there exists a map $\beta: S \to \Q(i)$ so that
\begin{equation}
\norm{v_0 - \sum_{\nu \in S} \beta(\nu)\psi(\nu)}_p < 2^{-N}.\label{ineq:beta}
\end{equation}
For each such a map $\beta$, define $F_\beta : \mathcal{B}^S \rightarrow \mathcal{B}$ by 
\[
F_\beta(\psi) = \sum_{\nu \in S} F_\beta(\nu) \psi(\nu).
\]
Then, $F_\beta$ is a computable map of $(\mathcal{B}^S)^\#$ into $\mathcal{B}^\#$; furthermore an index 
of $F_\beta$ can be computed from $S$ and $\beta$.  
Thus, by Proposition \ref{prop:preimage.c.e.open}, $V_\beta := F_\beta^{-1}(B(\mathbf{0}, 2^{-N}))$ is a c.e.
open subset of $(\mathcal{B}^S)^\#$; furthermore an index of $V_\beta$ can be computed from $\beta$, $N$, and $S$.  Since $\mathcal{S}_2 = \bigcup_\beta V_\beta$, it follows that $\mathcal{S}_2$ is a c.e. open subset of $(\mathcal{B}^S)^\#$ and that an index of $\mathcal{S}_2$ can be computed from $S$ and $N$.\\

(\ref{lm:c.e.open::itm:summative}):  
Now, suppose $S$ is an orchard.  
Let $\mathcal{S}_3$ denote the set of all $\psi \in (\mathcal{B}^S)^\#$ so that for every nonterminal node $\nu$ of $S$	
\[
\norm{\psi(\nu) - \sum_{\nu' \in \nu^+ \cap S} \psi(\nu')} < 2^{-N}
\]
where $\nu'$ ranges over all children of $\nu$ in $S$.
Fix a nonterminal node $\nu$ of $S$.  
Define a map $F_\nu : \mathcal{B}^S \rightarrow \mathcal{B}$ by 
\[
F_\nu(\psi) = \psi(\nu) - \sum_{\nu' \in \nu^+ \cap S} \psi(\nu').
\]
Then, $F_\nu$ is a computable map of $(\mathcal{B}^S)^\#$ into $\mathcal{B}$; furthermore an index of 
$F_\nu$ can be computed from $\nu$ and $S$.  
Thus, $W_\nu := F_\nu^{-1}(B(\mathbf{0}; 2^{-N}))$ is a c.e. open subset of $(\mathcal{B}^S)^\#$, and an 
index of $W_\nu$ can be computed from $S$, $N$, and $\nu$.  
Since $\mathcal{S}_3 = \bigcap_\nu W_\nu$, $\mathcal{S}_3$ is a c.e. open subset of 
$(\mathcal{B}^S)^\#$ and an index of $\mathcal{S}_3$ can be computed from $S$ and $N$.
\end{proof}

\begin{lemma}\label{lm:sep.anti.c.e.closed}
Suppose $p \geq 1$ is a computable real so that $p \neq 2$, and suppose $L^p(\Omega)^\#$ is a computable presentation of $L^p(\Omega)$.  Let $S$ be a finite orchard.  Then, the set of all separating antitone maps in $L^p(\Omega)^S$ is a c.e. closed subset of $(L^p(\Omega)^S)^\#$.  Furthermore, an index of this set can be computed from $S$.  
\end{lemma}

\begin{proof}
Let $\mathcal{H}$ denote the set of all separating antitone maps in $L^p(\Omega)^S$.  For each 
$\psi \in L^p(\Omega)^S$, let $f(\psi) = 2^p\sigma(\psi)$.  Thus, $f$ is a computable nonnegative function from 
$(L^p(\Omega)^S)^\#$ into $\C$; furthermore, an index of $f$ can be computed from $S$.  By Theorem \ref{thm:sigma.estimate}, $f(\psi) \geq d(\psi, \mathcal{H})$.  It follows from Corollary \ref{cor:sigma}, that 
$\mathcal{H} = f^{-1}[\{0\}]$.  So, by Proposition \ref{prop:bounding}, $\mathcal{H}$ is a c.e. closed 
subset of $(L^p(\Omega)^S)^\#$ and an index of $\mathcal{H}$ can be computed from $S$.
\end{proof}

\begin{lemma}\label{lm:approx.extension}
Suppose $p \geq 1$ is a computable real so that $p \neq 2$.  Let $\Omega$ be a separable measure space, and let $L^p(\Omega)^\#$ be a computable presentation of $L^p(\Omega)$.  Assume $\phi : S \rightarrow L^p(\Omega)$ is a computable partial disintegration of $L^p(\Omega)^\#$ whose success index is at least $n_1$.  Then, for every $k,n \in \N$, there is a computable partial disintegration $\psi$ of $L^p(\Omega)^\#$ so that 
$\dom(\psi) \supseteq S$, $\norm{\psi|_S - \phi}_S < 2^{-k}$, the success index of $\psi$ is at least $n$, 
and the success index of $\psi|_S$ is at least $n_1$.  Furthermore, $\dom(\psi)$ and an index of $\psi$ can be computed from $k,n$, and an index of $\phi$.
\end{lemma}

\begin{proof}
For the moment, fix a finite orchard $S'$.  Let $U_{S'}$ denote the set of all injective maps in $L^p(\Omega)^{S'}$.  Let $V_{S', n}$ denote the set of all maps in $S'$ whose success index is at least $n$.  
Let $\mathcal{H}_{S'}$ denote the set of all separating antitone maps in $L^p(\Omega)^{S'}$.  

By Lemma \ref{lm:c.e.open}, $U_{S'}$ and $V_{S', n}$ are c.e. open subsets of $(L^p(\Omega)^{S'})^\#$ and 
indices of these sets can be computed from $S'$, $n$.  By Lemma \ref{lm:sep.anti.c.e.closed}, $\mathcal{H}_{S'}$ is a c.e. closed subset of $(L^p(\Omega)^{S'})^\#$ and an index of $\mathcal{H}_{S'}$ can be computed from $S'$.

When $S' \supseteq S$, let $\pi_{S'}$ denote the canonical projection of $L^p(\Omega)^{S'}$ onto 
$L^p(\Omega)^S$, and let 
\[
C_{S'} = U_{S'} \cap V_{S',n} \cap \pi^{-1}_{S'}[B(\phi; 2^{-k}) \cap V_{S, n_1}] \cap \mathcal{H}_{S'}.
\]
By Theorem \ref{thm:extend.partial.disint}, there \emph{is} an $S'$ so that 
$C_{S'} \neq \emptyset$.  Such an $S'$ can be found by an effective search procedure.  By Proposition \ref{prop:comp.point}, 
$C_{S'}$ contains a computable vector $\psi$ of $(L^p(\Omega)^{S'})^\#$ and an index of $\psi$ can be computed from $k$, $n$, and an index of $\phi$.  
\end{proof}

We are now ready to prove Theorem \ref{thm:disint.comp}.

\begin{proof}[Proof of Theorem \ref{thm:disint.comp}:]
Suppose $L^p(\Omega)^\# = (L^p(\Omega), R)$.  

Set $S_0 = \{(0)\}$.  Since $\Omega$ is nonzero, $R(j_0) \neq \mathbf{0}$ for some $j_0$; such a number $j_0$ can be computed by a search procedure.  Set $\hat{\phi}_0((0)) = R(j_0)$.  
By Lemma \ref{lm:c.e.open} we can compute $k_0 \in \N$ so that every map in $B(\hat{\phi}_0; 2^{-k_0})$
is injective and never $0$.

It now follows from Lemma \ref{lm:approx.extension} that there is a sequence $\{\hat{\phi}_n\}_n$ of 
computable partial disintegrations of $L^p(\Omega)^\#$ and a computable sequence $\{k_n\}_n$ of 
nonnegative integers that have following properties. 
\begin{enumerate}
	\item An index of $\hat{\phi}_n$ and a canonical index of $\dom(\hat{\phi}_n)$ can be computed form $n$. 
	
	\item If $S_n =\dom(\phi_n)$, then $S_n \subseteq S_{n+1}$ and $\norm{ \hat{\phi}_{n+1}|_{S_n} - \hat{\phi}_n}_{S_n} < 2^{-(k_n + 1)}$. 
	
	\item Each map in $B(\hat{\phi}_n; 2^{-k_n})$ is injective, never zero, and has a success index that is at least $n$.
\end{enumerate}

So, let $\phi_{n,t} = \hat{\phi}_{t + n} |_{S_n}$ for all $n,t$.  It follows that $\{\phi_{n,t}\}_t$ is a computable
sequence of $(L^p(\Omega)^{S_n})^\#$; furthermore, an index of this sequence can be computed from $n$.  
It also follows that $\norm{\phi_{n,t+1} - \phi_{n,t}}_{S_n} < 2^{-(k_{n+t} + 1)}$.  Thus, by Proposition \ref{prop:eff.cauchy}, $\phi_n := \lim_t \phi_{n,t}$ is a computable vector of $(L^p(\Omega)^{S_n})^\#$; furthermore, an index of $\phi_n$ can be computed from $n$.  Also, $\norm{\hat{\phi}_n - \phi_n}_{S_n} \leq 2^{-k_n}$.  Thus, $\phi_n$ is a partial disintegration whose success index is at least $n$.  Since $S_n \subseteq S_{n+1}$, $\phi_{n,t+1} \subseteq \phi_{n+1,t}$.  Thus, $\phi_n \subseteq \phi_{n+1}$.  Let $\phi = \bigcup_n \phi_n$.

The only thing that prevents $\phi$ from being a disintegration is that $\lambda \not \in \dom(\phi)$.   
We fix this as follows.  Let $S = \dom(\phi)$.  For each $\nu \in S$, let 
\[
\psi(\nu) = 2^{-\nu(0)} \norm{\phi(\nu(0))}_p^{-1} \phi(\nu).
\]
Then, let 
\[
\psi(\lambda) = \sum_{\nu \in \N^1 \cap S} \psi(\nu).
\]
Since $S$ is computable, it follows that $\psi(\lambda)$ is a computable vector of $L^p(\Omega)^\#$.  It then follows that $\psi$ is a computable disintegration of $L^p(\Omega)^\#$.
\end{proof}

\section{A comparison of arguments for $\ell^p$ and $L^p(\Omega)$ spaces}\label{sec:comparison}

Here, we discuss why arguments previously used to show that certain $\ell^p$ spaces are not computably
categorical can not be applied to $L^p[0,1]$.  We then discuss why our techniques for $L^p$ spaces of 
non-atomic measure spaces can not be applied to $\ell^p$ spaces.

We begin by examining why arguments for $\ell^p$ spaces can not be generalized to $L^p[0,1]$.  As mentioned in Subsection \ref{subsec:survey.prior}, Pour-El and Richards proved that $\ell^1$ is not computably categorical.  Their proof rests on an observation about the extreme points of the unit ball in $\ell^1$.  However, 
the unit ball in $L^1[0,1]$ does not have extreme points.  Later, McNicholl showed that $\ell^p$ is computably categorical only when $p = 2$.  His proof utilizes the Banach-Lamperti characterization of the isometries of $\ell^p$, which extends to $L^p$ spaces of $\sigma$-finite spaces.  However, it also uses the fact that $\ell^p$ has a disjointly supported Schauder basis which $L^p[0,1]$ does not.

We now discuss why our arguments for $L^p$ spaces can not be applied to $\ell^p$ spaces.  
In particular, we look at the three key steps stated in Section \ref{sec:overview}.   
Theorems \ref{thm:lifting.comp} and \ref{thm:disint.comp} do not assume the underlying measure spaces 
are non-atomic.  But, Theorem \ref{thm:comp.disint.Lp01} does.  And, the construction in \cite{McNicholl.Stull.2016} shows that when $p \neq 2$ there is a computable presentation $\mathcal{B}^\#$ of $\ell^p$ and a computable disintegration $\phi$ of $\mathcal{B}^\#$ that is not computably isomorphic to any computable 
disintegration of the standard presentation of $\ell^p$.

\section{Relative computable categoricity}\label{sec:rcc}

We begin by defining what we mean by the diagram of a presentation of a Banach space.  Our approach parallels that in \cite{Greenberg.Knight.Melnikov.Turetsky.2016}.  

Suppose $\mathcal{B}$ is a Banach space and $\mathcal{B}^\# = (\mathcal{B}, R)$ is a presentation of 
$\mathcal{B}$.  
We define the \emph{diagram} of $\mathcal{B}^\#$ to be the set of all triples
$(v,r_0,r_1)$ so that $v$ is a rational vector of $\mathcal{B}^\#$, $r_0, r_1 \in \Q$, and 
$r_0 < \norm{v} < r_1$.  We denote the diagram of $\mathcal{B}^\#$ by $D(\mathcal{B}^\#)$.  
 
We say that a separable Banach space $\mathcal{B}$ is \emph{relatively computably categorical} if there is a Turing machine $M$ so that whenever $\mathcal{B}^\#$ and $\mathcal{B}^+$ are 
presentations of $\mathcal{B}$ and $\{d_n^\#\}_{n \in \N}$, $\{d_n^+\}_{n \in \N}$ are enumerations of $D(\mathcal{B}^\#)$, $D(\mathcal{B}^+)$ respectively, their join computes an isometric isomorphism of $\mathcal{B}^\#$ onto $\mathcal{B}^+$.  


Our proofs are sufficiently uniform to show the following.

\begin{theorem}\label{thm:rcc}
If $L^p(\Omega)$ is computably presentable, and if $\Omega$ is non-atomic and separable, then 
$L^p(\Omega)$ is relatively computably categorical.
\end{theorem}

\section{Computable measure spaces}\label{sec:comp.msr.spaces}

We first define what we mean by a computable presentation of a measure space.  Our approach parallels that in \cite{Ding.Weihrauch.Wu.2009}.

To begin, suppose $\mathcal{R}$ is a ring of sets.  A \emph{structure} on $\mathcal{R}$ is a map of $\N$ onto $\mathcal{R}$.  If $R$ is a structure on $\mathcal{R}$, the pair $(\mathcal{R}, R)$ is called a \emph{presentation} of $\mathcal{R}$.  

Now, suppose $\Omega = (X, \mathcal{M}, \mu)$ is a measure space.  A \emph{structure} on $\Omega$ is a 
structure on a ring that generates $\mathcal{M}$ and whose members have finite measure.  If 
$R$ is a structure on $\Omega$, then the pair $(\Omega, R)$ is called a presentation of $\Omega$.

Suppose $\mathcal{R}^\# = (\mathcal{R}, R)$ is a presentation of a ring of sets.  
We say that $\mathcal{R}^\#$ is a \emph{computable presentation of $\mathcal{R}$} if there are computable functions $f,g$ from $\N^2$ into $\N$ so that for all $m,n \in \N$
\begin{eqnarray*}
R(n) \cup R(m) & = & R(f(n,m))\mbox{, and}\\
R(n) - R(m) & = & R(g(m,n)).
\end{eqnarray*}

Let $\Omega$ be a measure space.  Suppose $R$ is a
structure on $\Omega$, and let $\mathcal{R} = \ran(R)$.  We say $(\Omega, R)$ is a \emph{computable presentation of $\Omega$} if $(\mathcal{R}, R)$ is a computable presentation 
of $\mathcal{R}$ and if $\mu(R(n))$ can be computed from $n$; that is if there is an algorithm that given $n, k \in \N$ as input computes a rational number $q$ so that $|q - \mu(R(n))| < 2^{-k}$.  

A measure space $\Omega = (X, \mathcal{M}, \mu)$ is said to be \emph{countably generated} if the $\sigma$-algebra $\mathcal{M}$ is generated by a countable collection of measurable sets each of which has finite measure; such a collection is said to \emph{generate} $\Omega$. 

We have two key results.

\begin{theorem}\label{thm:comp.str.cms}
If $R$ is a computable structure on a measure space $\Omega$, and if $D_R(n) = \chi_{R(n)}$ for all $n \in \N$, then $(L^p(\Omega), D_R)$ is a computable presentation of $L^p(\Omega)$ for every 
computable real $p \geq 1$.  
\end{theorem}

\begin{theorem}\label{thm:non.comp.msr.space}
There is a countably generated measure space $\Omega$ that does not have a computable presentation  
but so that $L^p(\Omega)$ has a computable presentation whenever $1 \leq p < \infty$ is computable.
\end{theorem}

To prove Theorem \ref{thm:comp.str.cms}, we need some preliminary material on measure spaces.
 It is well-known that every countably generated measure space is separable but not conversely.   In particular, the following is essentially Theorem A p. 168 of Halmos \cite{Halmos.1950}.

\begin{theorem}\label{thm:density.ring}
Suppose $\Omega$ is a measure space and that 
$\mathcal{G}$ is a countable set that generates $\Omega$.  Then, the ring generated by $\mathcal{G}$ is dense in $\mathcal{S}$.
\end{theorem}

\begin{corollary}\label{cor:structure.1}
Suppose $\Omega = (X, \mathcal{M}, \mu)$ is a measure space that is generated by a countable set $\mathcal{G}$, and let $\mathcal{R}$ denote the ring generated by $\mathcal{G}$.  Then, 
the linear span of $\{\chi_R\ :\ R \in \mathcal{R}\}$ is dense in $L^p(\Omega)$.
\end{corollary}

\begin{proof}
Let $\mathcal{R}' = \{\chi_R\ :\ R \in \mathcal{R}\}$.  
It follows from Theorem \ref{thm:density.ring} that $\chi_A$ lies in the subspace generated by $\mathcal{R}'$ whenever $A$ is a measurable set whose measure is finite.   Thus, $s$ belongs to the subspace generated by $\mathcal{R}$ whenever $s$ is a simple function whose support has finite measure.   
Since these functions are dense in $L^p(\Omega)$, so is the linear span of $\mathcal{R}'$.
\end{proof}

Suppose $\Omega^\# = (\Omega, R)$ is a presentation of $\Omega$.   Set $D_R(n) = \chi_{R(n)}$ for all $n \in \N$.  It follows from Theorem 
\ref{thm:density.ring} that $(L^p(\Omega), D_R)$ is a presentation of $L^p(\Omega)$ which we refer to as the \emph{induced presentation}.  We are now ready to prove Theorem \ref{thm:comp.str.cms}.

\begin{proof}[Proof of Theorem \ref{thm:comp.str.cms}:]
It follows from Corollary \ref{cor:structure.1} that the linear span of \\
$\{\chi_{R(n)}\ |\ n \in \N\}$ is dense in $L^p(\Omega)$.
Suppose $\alpha_0, \ldots, \alpha_M \in \Q(i)$.  For each $h \in \{0,1\}^{M+1}$ set:
\begin{eqnarray*}
S_h & =  & \bigcap_{h(j) = 1} R_j \cap \bigcap_{h(j) = 0} (X - R_j)\\
\beta_h & = & \sum_{h(j)=1} \alpha_j
\end{eqnarray*}
Since $R$ is a computable structure on $\Omega$, $\mu(S_h)$ can be computed uniformly from $h$.  
Note that $S_{h_1} \cap S_{h_2} = \emptyset$ whenever $h_1, h_2$ are distinct.  Since, 
\[
\sum_{n = 0}^M \alpha_n \chi_{R_n} = \sum_h \beta_h \chi_{S_h}
\]
it follows that
\[
\norm{\sum_{n = 0}^n \alpha_n \chi_{R_n}}_p^p = \sum_h |\beta_h|^p \mu(S_h).
\]
Thus, $\norm{\sum_{n = 0}^n \alpha_n \chi_{R_n}}_p$ can be computed uniformly from $M, \alpha_0, \ldots, \alpha_M$.
\end{proof}

To prove Theorem \ref{thm:non.comp.msr.space} we will need the following observation.

\begin{proposition}\label{prop:lsc.measure}
Suppose $\Omega=(X, \mathcal{M}, \mu)$ is a finite measure space and $\Omega^\#$ is a computable 
presentation of $\Omega$.  Then, $\mu(X)$ is a lower semi-computable real.
\end{proposition}

\begin{proof}
Suppose $\Omega^\# = (X, \mathcal{M}, \mu, R)$.  Thus, $X = \bigcup_n R(n)$.  Let: 
\[
F_n = R(n) - \bigcup_{m < n} R(m).
\]
Thus, $F_0, F_1, \ldots$ are pairwise disjoint and $X = \bigcup_n F_n$.  Furthermore, 
$\mu(F_n)$ is computable uniformly from $n$.  Thus, $\mu(X) = \sum_n \mu(F_n)$ is lower semi-computable.
\end{proof}

\begin{proof}[Proof of Theorem \ref{thm:non.comp.msr.space}:]
Let $X = [0,1]$.  Let $\mathcal{M}$ denote the $\sigma$-algebra generated by the dyadic 
subintervals of $[0,1]$.  Let $r$ be a positive real that is not lower semicomputable.  
Whenever $A \in \mathcal{M}$, let $\mu(A) = r \cdot m(A)$ where $m$ denotes Lebesgue measure.  
Thus, $\Omega := (X, \mathcal{M}, \mu)$ is a countably generated measure space.
Since $\mu(X) = r$, it follow from Proposition \ref{prop:lsc.measure} that $\Omega$ does not have a computable presentation.

Now, let $\{I_n\}_n$ be a standard enumeration of the dyadic subintervals of $[0,1]$, and let 
$R(n) = r^{-1} \chi_{I_n}$.  It follows that $D$ is a computable structure on $L^p(\Omega)$.
 For, $\norm{R(n)}_p = m(I_n)^{1/p}$, and each sum of the form $\sum_{n = 0}^M \alpha_n R(n)$ can be effectively rewritten as a sum of the form $\sum_{j = 0}^k \beta_j R(n_j)$ where 
$R(n_0), \ldots, R(n_k)$ are disjointly supported.
\end{proof}

\section{Conclusion}\label{sec:conclusion}

A long-term goal of analytic computable structure theory should be to classify the computably categorical Banach spaces.  Among the Banach spaces most encountered in practice in both pure and applied mathematics are the $L^p$ spaces.  So, a nearer-term subgoal is to classify the computably categorical
$L^p$ spaces.  As mentioned in the introduction,  all such spaces must be separable, and therefore their underlying measure spaces must be separable.  As shown in Section \ref{sec:comp.msr.spaces}, the computable presentability of an $L^p$ spaces does not imply the computable presentability of its underlying measure space.  

When analyzing the computable categoricity of $L^p$ spaces, it makes sense to divide them into the $L^p$ spaces of separable atomic spaces and the $L^p$ spaces of the separable non-atomic spaces.  Here, we have resolved the matter of the $L^p$ spaces of non-atomic measure spaces.  
That leaves the atomic spaces to be considered.  These can be divided into those that are purely atomic and those that are not.  Every separable atomic space has countably many atoms.  So, the purely atomic case has already been resolved; namely $\ell^p$ is computably categorical only when $p = 2$ and $\ell^p_n$ is computably categorical for all $p,n$ \cite{McNicholl.2015}, \cite{McNicholl.2016}.  So, only the $L^p$ spaces of non-atomic but not purely atomic spaces remain to be examined, and a future paper will do so.

One consequence of our main result is that when investigating the effective mathematics of 
$L^p[0,1]$, one need not be concerned about the choice of computable presentation as they all yield the same classes of computable points, sequences, and operators.  However, if one wishes to do some actual computing, then the question arises as to whether one presentation could be more advantageous than others.  Accordingly, we pose the question: if $p \geq 1$ is a computable real, and if $L^p[0,1]^\#$, $L^p[0,1]^+$ are two polynomial-time computable presentations of 
$L^p[0,1]$, does it follow that there is a polynomial-time computable isometric isomorphism of 
$L^p[0,1]^\#$ onto $L^p[0,1]^+$?

\section*{Acknowledgements}

The authors thank Ananda Weerasinghe for helpful discussions.  We are very grateful to the referee for many helpful and encouraging comments and for pointing out a few errors in some of the proofs.  The second author was supported in part by Simons Foundation Grant \# 317870.

\def\cprime{$'$}
\providecommand{\bysame}{\leavevmode\hbox to3em{\hrulefill}\thinspace}
\providecommand{\MR}{\relax\ifhmode\unskip\space\fi MR }
\providecommand{\MRhref}[2]{%
  \href{http://www.ams.org/mathscinet-getitem?mr=#1}{#2}
}
\providecommand{\href}[2]{#2}


\begin{thebibliography}{10}

\bibitem{Ash.Knight.2000}
C.~J. Ash and J.~Knight, \emph{Computable structures and the hyperarithmetical
  hierarchy}, Studies in Logic and the Foundations of Mathematics, vol. 144,
  North-Holland Publishing Co., Amsterdam, 2000.

\bibitem{Calvert.Cenzer.Harizanov.Morozov.2009}
Wesley Calvert, Douglas Cenzer, Valentina~S. Harizanov, and Andrei Morozov,
  \emph{Effective categoricity of abelian {$p$}-groups}, Ann. Pure Appl. Logic
  \textbf{159} (2009), no.~1-2, 187--197. \MR{2523717}

\bibitem{Cembranos.Mendoza.1997}
Pilar Cembranos and Jos{{\'e}} Mendoza, \emph{Banach spaces of vector-valued
  functions}, Lecture Notes in Mathematics, vol. 1676, Springer-Verlag, Berlin,
  1997. \MR{1489231}

\bibitem{Chajda.Halavs.Radomir.Kuhr.2007}
Ivan Chajda, Radom{\'{\i}}r Hala{\v{s}}, and Jan K{\"u}hr, \emph{Semilattice
  structures}, Research and Exposition in Mathematics, vol.~30, Heldermann
  Verlag, Lemgo, 2007.

\bibitem{Cooper.2004}
S.~Barry Cooper, \emph{Computability theory}, Chapman \& Hall/CRC, Boca Raton,
  FL, 2004.

\bibitem{Fokina.Harizanov.Melnikov.2014}
Ekaterina~B. Fokina, Valentina Harizanov, and Alexander~G. Melnikov,
  \emph{Computable model theory}, Turing's Legacy: Developments from Turing's
  Ideas in Logic (Rod Downey, ed.), Cambridge University Press, Cambridge,
  2014.

\bibitem{Froehlich.Shepherdson.1956}
A.~Fr\"ohlich and J.~C. Shepherdson, \emph{Effective procedures in field
  theory}, Philos. Trans. Roy. Soc. London. Ser. A. \textbf{248} (1956),
  407--432.

\bibitem{Goncharov.1977}
S.~S. Gon{\v{c}}arov, \emph{The number of nonautoequivalent
  constructivizations}, Algebra i Logika \textbf{16} (1977), no.~3, 257--282,
  377. \MR{516028}

\bibitem{Goncharov.1980.2}
\bysame, \emph{Autostability of models and abelian groups}, Algebra i Logika
  \textbf{19} (1980), no.~1, 23--44, 132. \MR{604656}

\bibitem{Dzgoev.Goncharov.1980}
S.~S. Gon{\v{c}}arov and V.~D. Dzgoev, \emph{Autostability of models}, Algebra
  i Logika \textbf{19} (1980), no.~1, 45--58, 132.

\bibitem{Goncharov.Lempp.Solomon.2003}
Sergey~S. Goncharov, Steffen Lempp, and Reed Solomon, \emph{The computable
  dimension of ordered abelian groups}, Adv. Math. \textbf{175} (2003), no.~1,
  102--143. \MR{1970243}

\bibitem{Goncharov.1975}
S.S. Goncharov, \emph{Autostability and computable families of
  constructivizations}, Algebra and Logic \textbf{17} (1978), 392--408, English
  translation.

\bibitem{Greenberg.Knight.Melnikov.Turetsky.2016}
N.~Greenberg, J.F. Knight, A.G. Melnikov, and D.~Turetsky, \emph{Uniform
  procedures in uncountable structures}, Preprint available at
  http://homepages.mcs.vuw.ac.nz/~greenberg/Papers/57-Syntax\underline{\
  }and\underline{\ }spaces.pdf.

\bibitem{Halmos.1950}
Paul~R. Halmos, \emph{Measure {T}heory}, D. Van Nostrand Company, Inc., New
  York, N. Y., 1950.

\bibitem{Harizanov.1998}
Valentina~S. Harizanov, \emph{Pure computable model theory}, Handbook of
  recursive mathematics, {V}ol.\ 1, Stud. Logic Found. Math., vol. 138,
  North-Holland, Amsterdam, 1998, pp.~3--114.

\bibitem{Hirschfeldt.Khoussainov.Shore.Slinko.2002}
Denis~R. Hirschfeldt, Bakhadyr Khoussainov, Richard~A. Shore, and Arkadii~M.
  Slinko, \emph{Degree spectra and computable dimensions in algebraic
  structures}, Ann. Pure Appl. Logic \textbf{115} (2002), no.~1-3, 71--113.
  \MR{1897023}

\bibitem{Lamperti.1958}
John Lamperti, \emph{On the isometries of certain function-spaces}, Pacific J.
  Math. \textbf{8} (1958), 459--466.

\bibitem{Levin.2016}
Oscar Levin, \emph{Computable dimension for ordered fields}, Arch. Math. Logic
  \textbf{55} (2016), no.~3-4, 519--534. \MR{3490918}

\bibitem{Malcev.1961}
A.~I. Mal\cprime~cev, \emph{Constructive algebras. {I}}, Uspehi Mat. Nauk
  \textbf{16} (1961), no.~3 (99), 3--60.

\bibitem{Malcev.1962}
\bysame, \emph{On recursive {A}belian groups}, Dokl. Akad. Nauk SSSR
  \textbf{146} (1962), 1009--1012.

\bibitem{McNicholl.Stull.2016}
T.H. McNicholl and D.~M. Stull, \emph{The isometry degree of a computable copy
  of $\ell^p$}, Submitted. Preprint available at
  http://arxiv.org/abs/1605.00641, 2016.

\bibitem{McNicholl.2015}
Timothy~H. McNicholl, \emph{A note on the computable categoricity of {$\ell^p$}
  spaces}, Evolving computability, Lecture Notes in Comput. Sci., vol. 9136,
  Springer, Cham, 2015, pp.~268--275.

\bibitem{McNicholl.2016}
Timothy~H. McNicholl, \emph{Computable copies of $\ell^p$}, Computability
  \textbf{6} (2017), no.~4, 391 -- 408.

\bibitem{Melnikov.2013}
Alexander~G. Melnikov, \emph{Computably isometric spaces}, J. Symbolic Logic
  \textbf{78} (2013), no.~4, 1055--1085.

\bibitem{Melnikov.Ng.2014}
Alexander~G. Melnikov and Keng~Meng Ng, \emph{Computable structures and
  operations on the space of continuous functions}, Fundamenta Mathematicae
  \textbf{233} (2014), no.~2, 1 -- 41.

\bibitem{Melnikov.Nies.2013}
Alexander~G. Melnikov and Andr{\'e} Nies, \emph{The classification problem for
  compact computable metric spaces}, The nature of computation, Lecture Notes
  in Comput. Sci., vol. 7921, Springer, Heidelberg, 2013, pp.~320--328.

\bibitem{Pour-El.Richards.1989}
Marian~B. Pour-El and J.~Ian Richards, \emph{Computability in analysis and
  physics}, Perspectives in Mathematical Logic, Springer-Verlag, Berlin, 1989.

\bibitem{Remmel.1981.2}
J.~B. Remmel, \emph{Recursive isomorphism types of recursive {B}oolean
  algebras}, J. Symbolic Logic \textbf{46} (1981), no.~3, 572--594. \MR{627907}

\bibitem{Sierpinski.1922}
W.~Sierpinski, \emph{Sur les fonctions d'ensemble additives et continues},
  Fundamenta Mathematicae \textbf{3} (1922), no.~1, 240--246.

\bibitem{Smith.1981}
Rick~L. Smith, \emph{Two theorems on autostability in {$p$}-groups}, Logic
  {Y}ear 1979--80 ({P}roc. {S}eminars and {C}onf. {M}ath. {L}ogic, {U}niv.
  {C}onnecticut, {S}torrs, {C}onn., 1979/80), Lecture Notes in Math., vol. 859,
  Springer, Berlin-New York, 1981, pp.~302--311. \MR{619876}

\bibitem{Steinberg.2017}
Florian Steinberg, \emph{Complexity theory for spaces of integrable functions},
  Log. Methods Comput. Sci. \textbf{13} (2017), no.~3, Paper No. 21, 39.

\bibitem{Ding.Weihrauch.Wu.2009}
Klaus Weihrauch, Yongcheng Wu, and Decheng Ding, \emph{Absolutely
  non-computable predicates and functions in analysis}, Math. Structures
  Comput. Sci. \textbf{19} (2009), no.~1, 59--71. \MR{2482020}

\end{thebibliography}
\end{document}